\documentclass{amsart}
\usepackage{amssymb, amsmath, amsfonts, amsthm, amscd, mathrsfs, mathtools, 
  stmaryrd, ulem, braket, latexsym, xparse
}
\usepackage[left=30mm,right=30mm]{geometry}

\usepackage{xcolor}
\definecolor{Myorange}{RGB}{100,0,230}

\usepackage{hyperref,aliascnt}
\hypersetup{colorlinks=true, citecolor=Myorange, linkcolor=Myorange, urlcolor=Myorange}
\renewcommand{\eqref}[1]{\textcolor{blue}{(\ref{#1})}}

\usepackage[T1]{fontenc}
\usepackage{roboto}

\usepackage{tikz, tikz-cd}
\usetikzlibrary{calc, matrix, arrows, decorations.markings}

\usepackage{enumitem, moreenum}
\makeatletter
\newcommand*{\@TrumpCount}[1]{\ensuremath{
  \ifcase #1\or\spadesuit\or\clubsuit\else \@ctrerr \fi\relax}}
\newcommand*{\TrumpCount}[1]{%
  \expandafter\@TrumpCount\csname c@#1\endcsname}
\AddEnumerateCounter{\TrumpCount}{\@TrumpCount}{2}
\AddEnumerateCounter{\fnsymbol}{\c@fnsymbol}{9}

\makeatother

\makeatletter

\newcommand{\Mytheorem}[2]{
  \newaliascnt{#1}{thm}\newtheorem{#1}[#1]{#2}\aliascntresetthe{#1}
  \expandafter\def\csname #1autorefname\endcsname{#2}
}
\NewDocumentCommand{\MyThm}{ >{\SplitList{,}} m }{\ProcessList{#1}{\MyThmEach}}
\def\MyThmEach#1{
  \def\@@MyThmEach{\relax #1\relax}
  \expandafter\ifx\@@MyThmEach\else
  \expandafter\@MyThmEach\@@MyThmEach\relax\fi
}
\def\@MyThmEach\relax#1=#2\relax{\Mytheorem{#1}{#2}}

\theoremstyle{plain}
\newtheorem{thm}{Theorem}[section]

\MyThm{prop=Proposition, cor=Corollary, lem=Lemma}

\newtheorem{thmA}{Theorem}[section]

\newaliascnt{corA}{thmA}
\newtheorem{corA}[corA]{Corollary}
\aliascntresetthe{corA}

\theoremstyle{definition}
\MyThm{defi=Definition, rem=Remark, notation=Notation,
  reconstruction=Reconstruction, assumption=Assumption,
}

\newtheorem*{exam*}{Example}
\newtheorem*{comment*}{Comment}
\newtheorem*{rrem*}{Remark}
\newtheorem*{defi*}{Definition}

\makeatother

\makeatletter

\renewcommand{\emptyset}{\varnothing}
\newcommand{\dfn}{:\overset{\mbox{{\rm\scriptsize def}}}{=}}

\renewcommand{\amalg}{\sqcup}
\newcommand{\red}{\mathrm{red}}

\NewDocumentCommand{\MyMathOpCommands}{ >{\SplitList{,}} m }{
  \ProcessList{#1}{\MyMathOpEach}
}
\def\MyMathOpEach#1{
  \def\@@MyMathOpEach{#1==\relax}
  \expandafter\ifx\@@MyMathOpEach\else
  \expandafter\@MyMathOpEach\@@MyMathOpEach\relax\fi
}
\def\@MyMathOpEach#1=#2=\relax{
  \ifx\relax#2\relax
  \expandafter\DeclareMathOperator\csname #1\endcsname{#1}\else
  \expandafter\DeclareMathOperator\csname #1\endcsname{\@@@MyMathOpEach #2}\fi
}
\def\@@@MyMathOpEach#1={#1}
\MyMathOpCommands{Hom, Isom, ISOM=\mathbf{Isom}, im=\mathrm{Im}, id, dom, codom, et=\mathrm{\acute{e}t}, pr,
coker, Spec, Aut, End, rk=\mathrm{Rank}, trdeg=\mathrm{tr.deg}}
\DeclareMathOperator*{\colim}{\mathrm{colim}}
\DeclareMathOperator*{\plim}{\mathrm{lim}}

\NewDocumentCommand{\MyMathRmCommands}{ >{\SplitList{,}} m }{
  \ProcessList{#1}{\MyMathRmEach}
}
\def\MyMathRmEach#1{
  \def\@@MyMathRmEach{#1==\relax}
  \expandafter\ifx\@@MyMathRmEach\else
  \expandafter\@MyMathRmEach\@@MyMathRmEach\relax\fi
}
\def\@MyMathRmEach#1=#2=\relax{
  \ifx\relax#2\relax
  \expandafter\newcommand\csname #1\endcsname{\mathrm{#1}}\else
  \expandafter\newcommand\csname #1\endcsname{\@@@MyMathRmEach #2}\fi
}
\def\@@@MyMathRmEach#1={#1}
\MyMathRmCommands{sh, gp, inte=\mathrm{int}, sat, tor, free, triv=\mathrm{triv}, op}

\newcommand{\univ}[1]{\mathbf{#1}}

\newcommand{\MyCatDef}[2][]{
  \ifx\relax#1\relax\expandafter\newcommand\csname #2\endcsname{\mathsf{#2}}
  \else\expandafter\newcommand\csname #2\endcsname{#1}\fi
  \expandafter\newcommand\csname #2U\endcsname{\csname#2\endcsname_{\univ{U}}}
  \expandafter\newcommand\csname #2V\endcsname{\csname#2\endcsname_{\univ{V}}}
  \expandafter\newcommand\csname #2W\endcsname{\csname#2\endcsname_{\univ{W}}}
  \expandafter\newcommand\csname #2Class\endcsname{\text{\rm\robotolight #2}}
}
\NewDocumentCommand{\MyCategories}{ >{\SplitList{,}} m }{
  \ProcessList{#1}{\MyCatDef}}
\MyCategories{Alg,Ring,Grp,Mon,Top,Cat,Grpd,Site,Sch,Mor,Sh}

\newcommand{\LSch}{\Sch^{\log}}

\makeatletter
\newcommand{\MyMathcal}[1]{\@tfor\Ch@r:=#1\do{%
  \expandafter\edef\csname mc\Ch@r\endcsname{\noexpand\mathcal{\Ch@r}}%
  \expandafter\edef\csname omc\Ch@r\endcsname{\noexpand\overline{\noexpand\mathcal{\Ch@r}}}%
  \expandafter\edef\csname scr\Ch@r\endcsname{\noexpand\mathscr{\Ch@r}}%
  \expandafter\edef\csname bf\Ch@r\endcsname{\noexpand\mathbf{\Ch@r}}%
  \expandafter\edef\csname bb\Ch@r\endcsname{\noexpand\mathbb{\Ch@r}}%
  \expandafter\edef\csname O\Ch@r\endcsname{\noexpand\mathcal{O}_{\Ch@r}}%
  \expandafter\edef\csname M\Ch@r\endcsname{\noexpand\mathcal{M}_{\Ch@r}}%
  \expandafter\edef\csname oM\Ch@r\endcsname{\noexpand\overline{\noexpand\mathcal{M}}_{\Ch@r}}
}}
\makeatother
\MyMathcal{ABCDEFGHIJKLMNOPQRSTUVWXYZ}
\newcommand{\A}{\bbA}

\newcommand{\F}{\bbF}
\newcommand{\G}{\bbG}
\newcommand{\N}{\bbN}

\renewcommand{\P}{\bbP}
\newcommand{\Q}{\bbQ}

\newcommand{\Z}{\bbZ}

\renewcommand{\OO}[1]{\mathcal{O}_{#1}}

\newcommand{\bbbullet}{\lozenge}
\newcommand{\bbullet}{\blacklozenge}

\MyMathRmCommands{qcpt, qsep, sep, Noeth=\mathrm{N}, lNoeth=\mathrm{lN}, ft, lft, lNn}
\newcommand{\rqqs}{\{ \red, \qcpt, \qsep, \sep \}}

\newcommand{\bcolim}{\colim^\bbullet}
\newcommand{\bplim}{\plim^\bbullet}
\newcommand{\bamalg}{\amalg^\bbullet}

\newcommand{\btimes}{\times^\bbullet}

\newcommand{\Schb}[1]{\Sch_{\bbullet/#1}}
\newcommand{\Schc}[1]{\Sch_{\bbbullet/#1}}
\newcommand{\LSchb}[1]{\LSch_{\bbullet/{#1}^{\log}}}
\newcommand{\LSchc}[1]{\LSch_{\bbbullet/{#1}^{\log}}}

\renewcommand{\textsf}[1]{{\small\robotolight #1}}

\makeatother

\newcommand\blfootnote[1]{%
  \begingroup
  \renewcommand\thefootnote{}\footnote{#1}%
  \addtocounter{footnote}{-1}%
  \endgroup
}

\begin{document}





\begin{center}
  \textbf{\Large{Category-theoretic Reconstruction of Log Schemes from}}

  \vspace{0.1cm}
  \textbf{\Large Categories of Reduced fs Log Schemes}

  \vspace{0.7cm}
  \textsc{Tomoki Yuji}

  \blfootnote{\textit{Email Address}: \href{mailto:math@yujitomo.com}{\texttt{math@yujitomo.com}}\\
    \indent\textit{Date}: \today 
  }
\end{center}


\begin{abstract}
  Let \(S^{\log}\) be a locally Noetherian fs log scheme
  and \(\bbullet/S^{\log}\) a set of properties of fs log schemes over \(S^{\log}\).
  In the present paper, we shall mainly be concerned with the properties ``reduced'', ``quasi-compact over \(S^{\log}\)'', ``quasi-separated over \(S^{\log}\)'', ``separated over \(S^{\log}\)'', and ``of finite type over \(S^{\log}\)''.
  We shall write \(\LSchb{S}\) for the full subcategory of the category of fs log schemes over \(S^{\log}\) determined by the fs log schemes over \(S^{\log}\) that satisfy every property contained in \(\bbullet/S^{\log}\).
  In the present paper, we discuss a purely category-theoretic reconstruction of the log scheme \(S^{\log}\) from the intrinsic structure of the abstract category \(\LSchb{S}\).
\end{abstract}

\setcounter{tocdepth}{1}
\tableofcontents


\section*[Introduction]{Introduction}
\hypertarget{introduction}{}

Throughout the present paper, we \textbf{fix} a Grothendieck universe \(\univ{U}\).
Let \[S^{\log}\] be a (\(\univ{U}\)-small) fs log scheme.
Write \(S\) for the underlying scheme of \(S^{\log}\).
Let \[\bbullet\] be a set of properties of morphisms of (\(\univ{U}\)-small) schemes (where we \textit{identify} properties of morphisms of schemes with certain full subcategories of the category of morphisms of schemes, cf. \hyperlink{notations}{Notations} \hyperlink{notations}{and} \hyperlink{notations}{Conventions} %
--- \hyperlink{notations and conventions -- properties}{Properties of Schemes and Log Schemes}%
).
We shall write \(\Sch_{/S}\) for the category of (\(\univ{U}\)-small) \(S\)-schemes, \(\Schb{S}\subset \Sch_{/S}\) for the full subcategory of objects of \(\Schb{S}\) that satisfy every property contained in \(\bbullet/S\) (cf. \hyperlink{notations}{Notations} \hyperlink{notations}{and} \hyperlink{notations}{Conventions} %
--- \hyperlink{notations and conventions -- properties}{Properties of Schemes and Log Schemes}%
), \(\LSch_{/S^{\log}}\) for the category of (\(\univ{U}\)-small) fs log schemes over \(S^{\log}\), and
\[\LSchb{S} \ \subset \ \LSch_{/S^{\log}}\]
for the full subcategory determined by the fs log schemes over \(S^{\log}\) whose underlying \(S\)-scheme is contained in \(\Schb{S}\).
In the present paper, we shall mainly be concerned with the situation where \(\bbullet/S^{\log}\) (cf. \hyperlink{notations}{Notations} \hyperlink{notations}{and} \hyperlink{notations}{Conventions} %
--- \hyperlink{notations and conventions -- properties}{Properties of Schemes and Log Schemes}%
) is contained in the following set of properties of log schemes over \(S^{\log}\):
\[\text{``red'', \ \ ``qcpt'', \ \ ``qsep'', \ \ ``sep'', \ \ ``ft''},\]
i.e.,
(the source scheme is) ``reduced'',
``quasi-compact over \(S^{\log}\)'',
``quasi-separated over \(S^{\log}\)'',
``separated over \(S^{\log}\)'', and
``of finite type over \(S^{\log}\)''.

In the present paper,
we consider the problem of reconstructing the log scheme \(S^{\log}\) from the intrinsic structure of the abstract category \(\LSchb{S}\).
The problem of reconstructing the scheme \(S\) from the intrinsic structure of the abstract category \(\Sch_{\bbullet/S}\) in the case where the elements of \(\bbullet/S\) amount essentially to the property of being ``finite \'etale over \(S\)'' is closely related to Grothendieck's anabelian conjectures and has been investigated by many mathematicians.
By contrast, the case where the elements of \(\bbullet/S\) differ substantially 
from the property of being ``finite \'etale over \(S\)'' has only been investigated to a limited degree.
In this case, there are some known results, mainly as follows:
In \cite[Section 1]{Mzk04}, Mochizuki gave a solution to this reconstruction problem in the case where \(S\) is locally Noetherian, and \(\bbullet = \{\mathrm{ft}\}\).
In \cite{deBr19}, van Dobben de Bruyn gave a solution to this reconstruction problem in the case where \(S\) is an arbitrary scheme, and \(\bbullet = \emptyset\).
The arguments in \cite[Section 1]{Mzk04} and \cite{deBr19} make essential use of the existence of non-reduced schemes.
On the other hand,
in \cite{YJ}, the author gave a solution to this reconstruction problem in the case where \(S\) is a locally Noetherian normal scheme, and one allows an arbitrary subset \(\bbullet \subset \rqqs\).

There are even fewer known results concerning the problem of reconstructing a log scheme \(S^{\log}\) from the intrinsic structure of the abstract category \(\LSchb{S}\).
In \cite{Mzk15} (and \cite[Section 2]{Mzk04}), S. Mochizuki proved that if \(\bbullet = \{\mathrm{ft}\}\), then a locally Noetherian fs log scheme \(S^{\log}\) may be reconstructed category-theoretically from the intrinsic structure of the abstract category \(\LSchb{S}\)
.
In \cite{HN}, Y. Hoshi and C. Nakayama gave a category-theoretic characterization of strict morphisms in the case where \(S^{\log}\) is locally Noetherian, and \(\bbullet = \{\mathrm{ft}\}\).
As discussed in \cite[Introduction]{HN}, the arguments of \cite{Mzk15} can be applied in more general situations where \(\bbullet = \{\mathrm{ft}\}\).
For instance, the condition assumed in \cite{Mzk15} that \(\bbullet = \{\mathrm{ft}\}\) may be replaced by the assumption that \(\bbullet = \{\mathrm{sep,ft}\}\) (for a detailed discussion, cf. \cite[Introduction]{HN}).
On the other hand, the proof given in \cite{Mzk15} is based on somewhat complicated combinatorial properties of monoids.
By constrast, while the arguments in \cite{HN} are somewhat more straightforward than the arguments of \cite{Mzk15}, the result of Hoshi and Nakayama depends essentially on the existence of non-separated log schemes in \(\LSch_{\{\mathrm{ft}\}/S^{\log}}\) (for a detailed discussion, cf. \cite[Introduction]{HN}).
In particular, the arguments in \cite{HN} cannot be applied in the situation, for instance, where \(\bbullet = \{\mathrm{ft, sep}\}\).
Here we note that the arguments in \cite{Mzk15} also make essential use of the existence of non-reduced schemes to give a characterization of ``SLEM'' morphisms (cf. \cite[Definition 2.1, Proposition 2.2]{Mzk15}).
Hence, the arguments in \cite{Mzk15} cannot be applied in the case, for instance, where \(\bbullet = \{\mathrm{ft, sep, red}\}\).

In the present paper, we give a \textbf{relatively simple} solution to this problem of reconstructing log structures in a situation that \textbf{generalizes} the situations discussed in \cite{Mzk04}, \cite{Mzk15}, and \cite{HN} to include the log scheme version of the situation discussed in \cite{YJ}.
Our main result is the following:

\begin{thmA}[{cf. \autoref{last cor}}]\label{main cor}
  Let \(S^{\log}, T^{\log}\) be locally Noetherian normal fs log schemes, \[\bbullet, \bbbullet \subset \{\mathrm{red, qcpt, qsep, sep, ft}\}\] [possibly empty] subsets, and \(F:\LSchb{S}\xrightarrow{\sim} \LSchc{T}\) an equivalence of categories.
  Assume that one of the following conditions \ref{last thm situation YJ intro}, \ref{last thm situation Mzk intro} holds:
  \begin{enumerate}[label=(\Alph*)]
    \item \label{last thm situation YJ intro}
    \(\bbullet, \bbbullet \subset \{\mathrm{red,qcpt,qsep,sep}\}\), and
    the underlying schemes of \(S^{\log}\) and \(T^{\log}\) are normal.
    \item \label{last thm situation Mzk intro}
    \(\bbullet = \bbbullet = \{\mathrm{ft}\}\).
  \end{enumerate}
  Then the following assertions hold:
  \begin{enumerate}
    \item 
    Let \(X^{\log}\in \LSchb{S}\) be an object.
    Then there exists an isomorphism of log schemes \(X^{\log}\xrightarrow{\sim} F(X^{\log})\) that is functorial with respect to \(X^{\log}\in \LSchb{S}\).
    \item \label{main 2 in intro}
    Assume that \(\bbullet = \bbbullet\).
    Then there exists a unique isomorphism of log schemes \(S^{\log}\xrightarrow{\sim} T^{\log}\) such that \(F\) is isomorphic to the equivalence of categories \(\LSchb{S} \xrightarrow{\sim} \LSchb{T}\) induced by composing with this isomorphism of log schemes \(S^{\log}\xrightarrow{\sim} T^{\log}\).
  \end{enumerate}
\end{thmA}


By combining the theory of \cite{YJ} with the above \autoref{main cor} \ref{main 2 in intro},
we conclude the following corollary (cf. \autoref{last cor A}):




\begin{corA}[{cf. \autoref{last cor A}}]
  \label{corA bb bbb}
  Let \(S^{\log}, T^{\log}\) be locally Noetherian normal fs log schemes and \[\bbullet, \bbbullet \subset \{\mathrm{red, qcpt, qsep, sep}\}\]subsets such that \(\{\mathrm{qsep, sep}\}\not\subset \bbullet\), and \(\{\mathrm{qsep, sep}\}\not\subset \bbbullet\).
  If the categories \(\LSchb{S}\) and \(\LSchc{T}\) are equivalent, then \(\bbullet = \bbbullet\), and \(S^{\log}\cong T^{\log}\).
\end{corA}


Our proof of \autoref{main cor} proceeds by giving
category-theoretic characterizations of
various properties of log schemes and morphisms of log schemes as follows:

\begin{itemize}
  \item
  In \autoref{section: quasi-int},
  we introduce some notions related to monoids and discuss various generalities that will play an important role in the theory of the present paper.
  In particular, we construct certain \textit{non-quasi-integral} push-out monoids (\autoref{cor: monoid qint}).
  Here, we recall that the notion of quasi-integral monoids was introduced by C. Nakayama (cf. \cite[Definition 2.2.4]{Nak}) to study the surjectivity of base-changes of morphisms of log schemes whose underlying morphism of sets is surjective.
  \item
  In \autoref{section: def log sch},
  we introduce some notions related to log schemes and discuss various properties of push-outs in the category of fs log schemes.
  In particular, we prove that certain push-outs exist in \(\LSchb{S}\) (cf. \autoref{cor: coprod exists}).
  \item
  In \autoref{section: Fs points},
  we give a category-theoretic characterization of the objects of \(\LSchb{S}\) whose underlying log scheme is an fs log point (cf. \autoref{prop: fs log pt}).
  This characterization also yields a category-theoretic characterization of the morphisms of \(\LSchb{S}\) whose underlying morphism of log schemes is isomorphic to a \textit{log residue field} (cf. \autoref{defi: log res fld}), i.e., the natural strict morphism that arises from the spectrum of the residue field at a point of the target log scheme (cf. \autoref{cor: log res fld}).
  \item
  In \autoref{section: str mor},
  we give a category-theoretic characterization of the morphisms in \(\LSchb{S}\) whose underlying morphism of log schemes is strict (cf. \autoref{strict general}).
  We then use this characterization and apply \cite[Corollary 4.11]{YJ} to obtain the first equality of \autoref{corA bb bbb} (cf. \autoref{bbul bbbul}).
  \item
  In \autoref{section: log like mor},
  we give a category-theoretic characterization of the morphisms of monoid objects in \(\LSchb{S}\) that represent the functor
  \begin{align*}
    \LSchb{S} &\to \Mor(\Mon) \\
    X^{\log} &\mapsto [\alpha_X(X): \Gamma(X,\mcM_X) \to \Gamma(X,\OX)],
  \end{align*}
  which arises from the log structures of the objects of \(\LSchb{S}\) (cf. \autoref{monoid object A}).
  We then use this characterization to prove the main theorem of the present paper (cf. \autoref{last cor}).
\end{itemize}


\addtocontents{toc}{\protect\setcounter{tocdepth}{0}}
\section*{Acknowledgements}

I would like to thank Professor Y. Hoshi, Professor S. Mochizuki, Professor A. Tamagawa, and Professor S. Yasuda for giving me advice on this paper and my research.


\addtocontents{toc}{\protect\setcounter{tocdepth}{1}}
\section*{Notations and Conventions}
\hypertarget{notations}{}

We shall use the notation \(\N\) to denote the additive monoid of non-negative rational integers \(n\geq 0\).
We shall use the notation \(\Z\) to denote the ring of rational integers.
We shall use the notation \(\Q\) to denote the field of fractions of \(\Z\).
Throughout the present paper, we \textbf{fix} a Grothendieck universe \(\univ{U}\).

\subsection*{Categories}
\hypertarget{notations and conventions -- categories}{}


Let \(\mcC\) be a category.
We shall write \(\mcC^{\mathrm{op}}\) for
the opposite category associated to \(\mcC\).
By a slight abuse of notation, we shall use the notation \(X\in \mcC\) to denote that \(X\) is an object of \(\mcC\).
We shall write \(\Mor(\mcC)\) for the category of morphisms of \(\mcC\), i.e., the category consisting of the following data:
\begin{itemize}
  \item
  An object of \(\Mor(\mcC)\) is a morphism in \(\mcC\).
  \item
  A morphism \([f:X\to Y] \to [g:X'\to Y']\) in \(\Mor(\mcC)\) is a pair \((h_X:X\to X', h_Y:Y\to Y')\) of morphisms in \(\mcC\) such that the following diagram commutes:
  \[\begin{tikzcd}
    X \ar[r,"h_X"]\ar[d,"f"'] & X'\ar[d,"g"] \\
    Y \ar[r,"h_Y"] & Y'.
  \end{tikzcd}\]
  \item
  The composite of morphisms is given on each component by composing morphisms in \(\mcC\).
\end{itemize}
Let \(X\in \mcC\) be an object.
We shall write \(\mcC_{/X}\) for the slice category of objects and morphisms equipped with a structure morphism to \(X\).
We shall write \(\mcC_{X/}\) for the over category of objects and morphisms equipped with a structure morphism from \(X\).

Let \(\mcD\subset \mcC\) be a full subcategory.
We shall say that \(\mcD\) is a \textbf{strictly full subcategory} of \(\mcC\) if \(\mcD\) is closed under isomorphism, i.e., for any object \(X\in \mcD\) and any isomorphism \(f:X\to Y\) in \(\mcC\), \(Y\) (hence also \(f\)) is contained in \(\mcD\).

\subsection*{Rings and Schemes}
\hypertarget{notations and conventions -- schemes}{}


We shall use the notation \(\Sch\) to denote the category of (\(\univ{U}\)-small) schemes.

Let \(f:Y\to X\) be a morphism of schemes.
We shall write \(\mathcal{O}_X\) for the structure sheaf of \(X\).
We shall write \(|X|\) for the underlying topological space of \(X\).
We shall write \(f^{\#}:\OX \to f_*\OY\) for the morphism of sheaves of rings on \(|X|\) which defines the morphism of schemes \(f:Y\to X\).
If \(|F|\subset |X|\) is a closed subset, then we shall write \(F_{\red}\) for the reduced closed subscheme of \(X\) determined by \(|F|\subset |X|\).
We shall write \(f_{\red}: X_{\red}\to Y_{\red}\) for the morphism induced by \(f\).
Let \(A\) and \(B\) be (commutative) rings (with unity).
If \(Y = \Spec(B)\), and \(X = \Spec(A)\), then we shall write \(f^{\#}: A\to B\) for the ring homomorphism induced by \(f\).
By a slight abuse of notation, if \(f^{\#}: A\to B\) is a ring homomorphism, then we shall use the notation \(f\) to denote the corresponding morphism of schemes \(\Spec(B) \to \Spec(A)\).
For any point \(x\in X\), we shall write \(k(x)\) for the residue field at \(x\in X\).

\subsection*{Properties of Schemes and Log Schemes}
\hypertarget{notations and conventions -- properties}{}


Let \(S^{\log}\) be a (\(\univ{U}\)-small) fs log scheme.
We shall write \(S\) for the underlying scheme of \(S^{\log}\).
We shall use the notation \[\LSch\] to denote the category of (\(\univ{U}\)-small) fs log schemes.
We shall write \(\Sch_{/S}\) for the category of (\(\univ{U}\)-small) \(S\)-schemes and \(\LSch_{/S^{\log}}\) for the category of (\(\univ{U}\)-small) fs log schemes over \(S^{\log}\).

  We shall refer to a strictly full subcategory (cf. \hyperlink{notations}{Notations and Conventions} --- \hyperlink{notations and conventions -- categories}{Categories}) of \(\Sch\), \(\Mor(\Sch)\), \(\LSch\), \(\Mor(\LSch)\), \(\Sch_{/S}\), and \(\LSch_{/S^{\log}}\) as a \textbf{property} of (\(\univ{U}\)-small) schemes, morphisms of (\(\univ{U}\)-small) schemes, (\(\univ{U}\)-small) fs log schemes, morphisms of (\(\univ{U}\)-small) fs log schemes, (\(\univ{U}\)-small) \(S\)-schemes, and (\(\univ{U}\)-small) fs log schemes over \(S^{\log}\), respectively.

Let \[\bbullet\] be a (not necessarily \(\univ{U}\)-small) set of properties of morphisms of (\(\univ{U}\)-small) schemes.
For any property \(\mathtt{P}\in \bbullet\), write \(\mathtt{P}/S\subset \Sch_{/S}\) for the full subcategory consisting of \(S\)-schemes whose structure morphism is contained in the full subcategory \(\mathtt{P}\subset \Mor(\Sch)\) and \(\mathtt{P}/S^{\log}\subset \LSch_{/S^{\log}}\) for the full subcategory consisting of fs log schemes over \(S^{\log}\) whose underlying \(S\)-scheme is contained in \(\mathtt{P}/S\).
By a slight abuse of notation, we shall write \(\bbullet/S \dfn \left\{\mathtt{P}/S \,\middle|\, \mathtt{P}\in \bbullet\right\}\) and
\(\bbullet/S^{\log}\dfn \left\{\mathtt{P}/S^{\log} \,\middle|\,\mathtt{P}\in \bbullet\right\}\).
In the present paper, we shall mainly be concerned with the situation where
\[\bbullet \ \subset \ \{ \ \text{red, \ qcpt, \ qsep, \ sep, \ ft} \ \},\]
and
\begin{itemize}
  \item
  ``red'' denotes the strictly full subcategory of \(\Mor(\Sch)\) consisting of morphisms whose source scheme is reduced,
  \item
  ``qcpt'' denotes the strictly full subcategory of \(\Mor(\Sch)\) consisting of quasi-compact morphisms,
  \item
  ``qsep'' denotes the strictly full subcategory of \(\Mor(\Sch)\) consisting of quasi-separated morphisms,
  \item
  ``sep'' denotes the strictly full subcategory of \(\Mor(\Sch)\) consisting of separated morphisms, and
  \item
  ``ft'' denotes the strictly full subcategory of \(\Mor(\Sch)\) consisting of morphisms of finite type.
\end{itemize}
Let \(f:X\to Y\) be a morphism of schemes.
Then we shall say that \(f\) \textbf{satisfies} every property contained in \(\bbullet\) if for any \(\mathtt{P}\in \bbullet\), \(f\in \mathtt{P}\subset \Mor(\Sch)\).
Hence, in particular, if \(\bbullet = \emptyset\), then every morphism of schemes satisfies every property contained in \(\bbullet\).
If \(\bbullet = \{\mathtt{P}\}\), and \(f\) satisfies every property contained in \(\bbullet\), then we shall say that \(f\) \textbf{satisfies} the property \(\mathtt{P}\).

We shall write \(\Sch_{\bbullet/S}\subset \Sch_{/S}\) for the full subcategory consisting of \(S\)-schemes whose structure morphisms satisfy every property contained in \(\bbullet\) and
\[\LSchb{S} \ \subset \ \LSch_{/S^{\log}}\]
for the full subcategory consisting of fs log schemes over \(S^{\log}\) whose underlying structure morphism of schemes is contained in \(\Schb{S}\).
Thus, if \(\bbullet = \emptyset\), then \(\LSchb{S} = \LSch_{/S^{\log}}\).

We shall write \(\bplim, \bcolim, \btimes, \bamalg\)
for the (inverse) limit, colimit, fiber product, and push-out in \(\LSchb{S}\) (if these exist in \(\LSchb{S}\)).
We shall write \(\plim, \colim, \times, \amalg\)
for the (inverse) limit, colimit, fiber product, and push-out in \(\LSch_{/S^{\log}}\) (if these exist in \(\LSch_{/S^{\log}}\)).
By a slight abuse of notation,
we shall use the notation \(\emptyset\) to denote the empty log scheme. 

  Let 
  \begin{itemize}
    \item \(S^{\log}\) be a (\(\univ{U}\)-small) locally Noetherian fs log scheme,
    \item \(\bbullet\subset \{\mathrm{red}, \mathrm{qcpt}, \mathrm{qsep}, \mathrm{sep}, \mathrm{ft}\}\) a subset,
    \item \(X^{\log}\in \LSchb{S}\) an object,
    \item \(\mathtt{P}\) a property of (\(\univ{U}\)-small) fs log schemes over \(S^{\log}\),
    \item \(f^{\log}\) a morphism in \(\LSchb{S}\), and
    \item \(\mathtt{Q}\) a property of morphisms of (\(\univ{U}\)-small) fs log schemes over \(S^{\log}\).
  \end{itemize}
  Then we shall say that
  \begingroup
  \addtolength\leftmargini{-0.1in}
  \begin{quote}
    \textbf{the property that} \(X^{\log}\) satisfies \(\mathtt{P}\) \textbf{may be characterized category-theoretically} from the data \((\LSchb{S}, X^{\log})\)
  \end{quote}
  \endgroup\noindent
  if for any object \(Y^{\log}\in \LSchb{S}\), any (\(\univ{U}\)-small) locally Noetherian fs log scheme \(T^{\log}\),
  any subset \(\bbbullet\subset \{\mathrm{red}, \mathrm{qcpt}, \mathrm{qsep}, \mathrm{sep}, \mathrm{ft}\}\), and
  any equivalence \(F:\LSchb{S}\xrightarrow{\sim}\LSchc{T}\),
  it holds that
  \begin{center}
    \(Y^{\log}\) satisfies \(\mathtt{P}\)
    \(\iff\) \(F(Y^{\log})\) satisfies \(\mathtt{P}\).
  \end{center}
  We shall say that
  \begingroup
  \addtolength\leftmargini{-0.1in}
  \begin{quote}
    \textbf{the property that} \(f^{\log}\) satisfies \(\mathtt{Q}\) \textbf{may be characterized category-theoretically} from the data \((\LSchb{S}, f^{\log})\)
  \end{quote}
  \endgroup\noindent
  if for any morphism \(g^{\log}\) in \(\LSchb{S}\), any (\(\univ{U}\)-small) locally Noetherian fs log scheme \(T^{\log}\),
  any subset \(\bbbullet\subset \{\mathrm{red}, \mathrm{qcpt}, \mathrm{qsep}, \mathrm{sep}, \mathrm{ft}\}\), and
  any equivalence \(F:\LSchb{S}\xrightarrow{\sim}\LSchc{T}\),
  it holds that
  \begin{center}
    \(f^{\log}\) satisfies \(\mathtt{Q}\)
    \(\iff\) \(F(f^{\log})\) satisfies \(\mathtt{Q}\).
  \end{center}

\section{Quasi-integral Monoids}
\label{section: quasi-int}

In this section,
we discuss various generalities concerning monoids that
will play an important role in the theory of the present paper.
Our main result (\autoref{cor: monoid qint}) concerns the construction of an extension of sharp fs monoids that satisfies a certain non-quasi-integrality property.

\begin{defi}\label{defi monoids}
  Let \(M\) be a (commutative) monoid.
  \begin{enumerate}
    \item We shall write \(\Mon\) for the category of (\(\univ{U}\)-small) monoids and \(\mathsf{Ab}\) for the category of (\(\univ{U}\)-small) abelian groups.
    \item
    We shall write \(M^{\gp}\) for the groupification of \(M\).
    \item
    We shall write \(M^{\times}\) for the unit group of \(M\).
    \item
    We shall write \(M^{\inte}\) for the image of the natural morphism \(M\to M^{\gp}\).
    \item \label{saturation defi}
    We shall write
    \[
    M^{\sat} \dfn \left\{ m\in M^{\gp} \, \middle| \, \exists a\in \N\setminus \{0\}, am \in M^{\inte}\right\}.
    \]
    \item
    We shall say that \(M\) is \textbf{sharp}
    if \(M^{\times} = 0\).
    \item
    We shall say that \(M\) is \textbf{integral}
    if the natural morphism \(M\to M^{\gp}\) is injective.
    If \(M\) is an integral monoid,
    then we regard \(M\) as a submonoid of \(M^{\gp}\) via
    the natural injection \(M\hookrightarrow M^{\gp}\).
    We shall write \(\Mon^{\inte}\subset \Mon\) for the full subcategory of \(\Mon\) consisting of integral monoids.
    \item
    We shall say that \(M\) is \textbf{saturated}
    if \(M\) is integral, and, moreover,
    for any element \(m\in M^{\gp}\), if there exists a positive integer \(a\in \N\setminus \{0\}\) such that \(am\in M\), then \(m\in M\).
    We shall write \(\Mon^{\sat}\subset \Mon\) for the full subcategory of \(\Mon\) consisting of saturated monoids.
    \item
    We shall say that \(M\) is \textbf{finitely generated}
    if there exist an integer \(a\in \N\) and
    a surjection \(\N^{\oplus a} \to M\) of monoids.
    \item \label{monoid generated by S}
    Let \(S\subset M\) be a subset, \(k\in \N\setminus \{0\}\) an integer, and \(m_1,...,m_k\in M\) elements.
    We shall write \(\left< S \right>\subset M\) for the submonoid of \(M\) generated by \(S\), i.e., the smallest submonoid of \(M\) that contains \(S\).
    We shall write \(\left< S, m_1,\cdots,m_k \right>\dfn \left< S \cup\{m_1,\cdots,m_k\}\right>\).
    \item
    We shall say that \(M\) is \textbf{fine}
    if \(M\) is integral and finitely generated.
    \item
    We shall say that \(M\) is \textbf{fs}
    if \(M\) is saturated and finitely generated.
  \end{enumerate}
\end{defi}

\begin{rem}\label{rem: int and sat monoids}
  \
  \begin{enumerate}
    \item \label{enumi: int rem: int and sat monoids}
    The functor \((-)^{\inte}:\Mon \to \Mon^{\inte}\)
    is a left adjoint functor to the natural inclusion
    \(\Mon^{\inte} \subset \Mon\).
    In particular,
    the limit of any diagram of integral monoids in \(\Mon\) is integral.
    \item \label{enumi: sat rem: int and sat monoids}
    The functor \((-)^{\sat}:\Mon \to \Mon^{\sat}\)
    is a left adjoint functor to the natural inclusion
    \(\Mon^{\sat} \subset \Mon\).
    In particular,
    the limit of any diagram of saturated monoids in \(\Mon\) is saturated.
    \item \label{enumi: fg saturation is fg}
    Let \(M\) be a fine monoid and \(K\) a field.
    Then observe that the natural inclusion of \(K\)-subalgebras
    \(K[M] \hookrightarrow K[M^{\sat}]\) of \(K[M^{\gp}]\)
    is integral and induces an isomorphism on quotient fields,
    hence, by a well-known result in commutative algebra, is \textit{finite}.
    This \textit{finiteness} of \(K[M^{\sat}]\) as a \(K[M]\)-module implies that
    \(M^{\sat}\) is a \textit{finitely generated monoid}.
  \end{enumerate}
\end{rem}

\begin{defi}
  Let \(f:M\to N\) be a morphism of monoids.
  \begin{enumerate}
    \item
    We shall write \(f^{\gp}:M^{\gp}\to N^{\gp}\) for the morphism induced on 
    groupifications by \(f\).
    \item
    We shall say that \(f\) is \textbf{local}
    if \(f^{-1}(N^{\times}) = M^{\times}\).
  \end{enumerate}
\end{defi}

\begin{lem}\label{lem: sharp fs generate}
  Let \(M\) be a sharp saturated monoid and \(n\in M^{\gp}\setminus M\).
  Then \(\left< M, -n \right>^{\sat}\) is sharp.
\end{lem}

\begin{proof}
  Write \(N \dfn \left< M, -n \right>^{\sat}\).
  Let \(\tilde{n}\in N^{\times}\) be an element.
  Then \(\tilde{n}, -\tilde{n}\in N^{\times}\subset N = \left< M, -n \right>^{\sat}\).
  Hence, by \autoref{defi monoids} \ref{saturation defi}, there exist \textbf{positive} integers \(a_1,a_2 \geq 1\) such that \(a_1\tilde{n}, -a_2\tilde{n}\in \left< M, -n \right>\).
  Moreover, by \autoref{defi monoids} \ref{monoid generated by S}, there exist non-negative integers \(b_1, b_2\in \N\), and elements \(m_1,m_2\in M\) such that
  \[
    a_1\tilde{n} = m_1 - b_1n, \ \ \text{and} \ \
    -a_2\tilde{n} = m_2 - b_2n.
  \]
  Thus it holds that \(a_2(m_1 - b_1n) + a_1(m_2 - b_2n) = a_2a_1\tilde{n} - a_1a_2\tilde{n} = 0\),
  hence \[(a_2b_1 + a_1b_2)n = a_2m_1 + a_1m_2 \in M.\]
  Since \(n\in M^{\gp}\setminus M\), and \(M\) is saturated, it holds that \(a_2b_1 + a_1b_2 = 0\).
  Since \(a_1, a_2\geq 1\), and \(b_1,b_2\in \N\), it holds that \(b_1 = b_2 = 0\).
  Thus \(a_1\tilde{n} = m_1\in M\), and \(-a_2\tilde{n} = m_2\in M\).
  Since \(a_1,a_2 \geq 1\), and \(M\) is saturated, it holds that \(\tilde{n}, -\tilde{n}\in M\).
  Since \(M\) is sharp, \(\tilde{n} = 0\).
  This completes the proof of \autoref{lem: sharp fs generate}.
\end{proof}



\begin{cor}\label{cor: sharp fs generate n elements}
  Let \(M\) be a monoid, \(N\) a sharp fs monoid, and
  \(i_1,i_2:N\to M\) morphisms of monoids.
  Assume that neither \(i_1\) nor \(i_2\) is injective.
  Then there exist elements \(n_1,n_2\in N^{\gp}\setminus N\)
  such that \(\left< N,-n_1,-n_2\right>^{\sat}\) is a sharp fs monoid, and
  \(i_1^{\gp}(n_1) = i_2^{\gp}(n_2) = 0\).
\end{cor}

\begin{proof}
  Since \(i_1\) is not injective, \(\ker(i_1^{\gp})\neq 0\).
  Since \(N\) is sharp, \(\ker(i_1^{\gp}) \not\subset N\).
  Let \(n_1\in \ker(i_1^{\gp})\setminus N\) be an element.
  Then, by \autoref{rem: int and sat monoids} \ref{enumi: fg saturation is fg} and \autoref{lem: sharp fs generate},
  \(L_1 \dfn \left< N,-n_1\right>^{\sat}\) is a sharp fs monoid.
  Since \(n_1\in \ker(i_1^{\gp})\),
  \(i_1^{\gp}(n_1) = 0\).

  Since \(i_2\) is not injective, \(\ker(i_2^{\gp})\neq 0\).
  Since \(L_1\) is sharp, \(\ker(i_2^{\gp}) \not\subset L_1\).
  Let \(n_2\in \ker(i_2^{\gp})\setminus L_1\) be an element.
  Then, by \autoref{rem: int and sat monoids} \ref{enumi: fg saturation is fg} and \autoref{lem: sharp fs generate},
  \(\left< L_1,-n_2\right>^{\sat} = \left< N,-n_1,-n_2\right>^{\sat}\)
  is a sharp fs monoid.
  Since \(n_2\in \ker(i_2^{\gp})\),
  \(i_2^{\gp}(n_2) = 0\).
  This completes the proof of \autoref{cor: sharp fs generate n elements}.
\end{proof}

\begin{defi}
  We shall say that a monoid \(M\) is
  \textbf{quasi-integral} if
  for any \(m,n\in M\),
  the equality \(m+n = m\) implies that \(n = 0\), or, equivalently, any element of \(M\) that becomes trivial in \(M^{\gp}\) is trivial in \(M\).
\end{defi}




In \autoref{appendix Nak Lem}, we prove an extension of the following lemma to the case where
\(M\), \(L\) are quasi-integral, and \(N\) is an arbitrary monoid (cf. \autoref{Nak extension}).

\begin{lem}[{cf. \cite[Lemma 2.2.6 (i)]{Nak}, \autoref{Nak extension}}]
  \label{lem: qint Nak in other words}
  Let \(L,M,N\) be sharp fs monoids and
  \(f:N\to M, g:N\to L\) local morphisms of monoids.
  Write \(P \dfn M\amalg_N L\).
  Then the following assertions are equivalent:
  \begin{enumerate}
    \item \label{Nak P is qint}
    \(P\) is quasi-integral.
    \item \label{Nak qint condition}
    For any element \(n\in N^{\gp}\),
    if \(f^{\gp}(n) \in M\), and \(-g^{\gp}(n)\in L\),
    then \(f^{\gp}(n) = 0\), and \(g^{\gp}(n) = 0\).
  \end{enumerate}
\end{lem}

\begin{proof}
  \autoref{lem: qint Nak in other words} follows immediately from \cite[Lemma 2.2.6 (i)]{Nak}. 
%
\end{proof}

\begin{cor}\label{cor: monoid qint}
  Let \(M,N\) be sharp fs monoids and
  \(i_1,i_2: N\to M\) local morphisms of monoids.
  Assume that neither \(i_1\) nor \(i_2\) is injective.
  Then there exists a sharp fs monoid \(L\) such that
  \(N\subset L\subset N^{\gp}\), and
  neither \(M\amalg_{i_1,N}L\) nor \(M\amalg_{i_2,N}L\) is quasi-integral.
\end{cor}

\begin{proof}
  Since neither \(i_1\) nor \(i_2\) is injective, it follows from \autoref{cor: sharp fs generate n elements} that there exist elements \(n_1,n_2\in N^{\gp}\setminus N\) such that \(L \dfn \left< N,-n_1,-n_2\right>^{\sat}\) is a sharp fs monoid, and \(i_1^{\gp}(n_1) = i_2^{\gp}(n_2) = 0\).
  For each \(k\in \{1,2\}\),
  since \(M,N,L\) are sharp fs monoids, \(i_k:N\to M\) is local, the inclusion morphism \(N\hookrightarrow L\) is local, \(i_k^{\gp}(n_k) = 0 \in M\), \(-n_k \in L\), and \(n_k \neq 0\), it follows from \autoref{lem: qint Nak in other words} that \(M\amalg_{i_k,N}L\) is not quasi-integral.
  This completes the proof of \autoref{cor: monoid qint}.
\end{proof}

\section{Some Remarks on Log Schemes}
\label{section: def log sch}

In this section,
we introduce some notions related to log schemes and
give proofs of several elementary results on log schemes.
In particular, we prove that certain push-outs exist in \(\LSchb{S}\) (cf. \autoref{cor: coprod exists}).

\begin{defi}\label{defi log sch}
  Let \(X^{\log}, Y^{\log}\) be log schemes and \(f^{\log} : X^{\log} \to Y^{\log}\) a morphism of log schemes
  (cf. \cite[Section 1]{Kato}).
  \begin{enumerate}
    \item
    We shall write \(X\) for the underlying scheme of the log scheme \(X^{\log}\).
    We shall write \(f\) for the underlying morphism of schemes \(f:X\to Y\) of the morphism of log schemes \(f^{\log}\).
    \item
    We shall write \(\alpha_X:\mcM_X\to \mcO_X\) for the morphism of sheaves of monoids on the \'etale site of \(X\) which defines the log structure of \(X^{\log}\).
    We shall write \(\oMX \dfn \MX/(\alpha_X^{-1}\OX^{\times})\).
    \item
    We shall write \(f^{\flat}:f^{-1}\mcM_Y\to \mcM_X\) for the morphism of sheaves of monoids on the \'etale site of \(X\) which defines the morphism of log schemes \(f^{\log}: X^{\log}\to Y^{\log}\).
    We shall write \(\bar{f}^{\flat}:f^{-1}\oMY\to \oMX\) for
    the morphism of sheaves of monoids on the \'etale site of \(X\) induced by \(f^{\flat}:f^{-1}\MY\to \MX\).
    \item
    We shall write \(f^*\alpha_Y:f^*\mcM_Y\to \mcO_X\) for the log structure on \(X\) determined by the pull-back of the log structure \(\alpha_Y:\mcM_Y\to \mcO_Y\) on \(Y\) via \(f\).
    \item \label{Zlog defi}
    Let \(|Z|\subset |X|\) be a closed subspace.
    Then we shall write \(Z^{\log}_{\red}\) for the log scheme
    whose underlying scheme is \(Z_{\red}\), and
    whose log structure is induced from \(X^{\log}\)
    via the natural closed immersion \(Z_{\red}\hookrightarrow X\).
    We shall write \((X^{\log})_{\red} \dfn X^{\log}_{\red}\).
  \end{enumerate}
\end{defi}

\begin{defi}\label{defi: log morphs}
  Let \(X^{\log}, Y^{\log}\) be log schemes,
  \(f^{\log}: X^{\log} \to Y^{\log}\) a morphism of log schemes,
  \(\mathtt{P}\) a property of schemes
  [such as ``quasi-compact'', ``separated'', ``quasi-separated'',
  ``reduced'', ``connected''], and
  \(\mathtt{Q}\) a property of morphisms of schemes
  [such as ``quasi-compact'', ``separated'', ``quasi-separated'',
  ``of finite type'', ``\'{e}tale''].
  \begin{enumerate}
    \item \label{enumi: Xlog satisfies P}
    We shall say that \(X^{\log}\) satisfies \(\mathtt{P}\) if
    the underlying scheme \(X\) satisfies \(\mathtt{P}\).
    \item \label{enumi: flog satisfies Q}
    We shall say that \(f^{\log}\) satisfies \(\mathtt{Q}\) if
    the underlying morphism of schemes \(f\) satisfies \(\mathtt{Q}\).
    \item
    We shall say that \(f^{\log}\) is \textbf{strict}
    (or, in the terminology of \cite{Mzk04}, \textbf{scheme-like}),
    if the morphism of sheaves of monoids \(f^*\mcM_Y\to \mcM_X\) on the \'etale site of \(X\) induced by \(f^{\flat}\) is an isomorphism.
    \item
    We shall say that \(f^{\log}\) is a \textbf{strict closed immersion} if
    \(f^{\log}\) is strict and a closed immersion.
    \item
    We shall say that \(f^{\log}\) is a \textbf{strict open immersion} if
    \(f^{\log}\) is strict and an open immersion.
    \item \label{enumi: str et neighb defi: log morphs}
    Let \(\bar{x}\to X\) be a geometric point.
    We shall say that a morphism of log schemes
    \(i^{\log}: U^{\log}\to X^{\log}\) is a
    \textbf{strict \'{e}tale neighborhood} of \(\bar{x}\) if \(i:U\to X\) is an \'{e}tale neighborhood of \(\bar{x}\), and \(i^{\log}\) is strict.
  \end{enumerate}
\end{defi}

\begin{defi}
  Let \(X^{\log}\) be a log scheme.
  \begin{enumerate}
    \item
    We shall say that the log structure of \(X^{\log}\) is \textbf{integral} if for any geometric point \(\bar{x} \to X\),
    the stalk \(\mcM_{X,\bar{x}}\) of \(\mcM_X\) at \(\bar{x}\to X\) is an integral monoid.
    \item
    We shall say that the log structure of \(X^{\log}\) is \textbf{saturated}
    if for any geometric point \(\bar{x} \to X\),
    the stalk \(\mcM_{X,\bar{x}}\) of \(\mcM_X\) at \(\bar{x}\to X\) is a saturated monoid.
    \item
    We shall say that \(X^{\log}\) \textbf{has a global chart}
    if there exist a monoid \(M\) and a morphism of monoids \(M\to \Gamma(X,\OX)\) such that the log structure of \(X^{\log}\) is isomorphic to the log structure associated to the adjoint morphism of sheaves of monoids \(\tilde{M} \to \OX\) (cf. \cite[(1.1), (1.3)]{Kato}), where \(\tilde{M}\) is the constant sheaf on the \'etale site of \(X\) associated to \(M\).
    In this situation, we shall refer to the morphism of monoids \(M\to \Gamma(X,\OX)\), or, equivalently, the morphism of sheaves of monoids \(\tilde{M}\to \OX\) on \(X\), as a \textbf{global chart} of \(X^{\log}\).
    \item
    Let \(\bar{x}\to X\) be a geometric point.
    Then we shall say that \(X^{\log}\) has a \textbf{chart} at \(\bar{x}\) if there exists a strict \'{e}tale neighborhood (cf. \autoref{defi: log morphs} \ref{enumi: str et neighb defi: log morphs}) \(U^{\log}\to X^{\log}\) of \(\bar{x}\) such that the log scheme \(U^{\log}\) has a global chart \(M\to \Gamma(U,\OU)\).
    In this situation, we shall say that \((U^{\log}\to X^{\log}, M\to \Gamma(U,\OU))\) is a \textbf{chart} at \(\bar{x}\).
    \item
    Assume that the log structure of \(X^{\log}\) is integral.
    Then we shall say that \(X^{\log}\) is a \textbf{fine} log scheme if for any geometric point \(\bar{x} \to X\), \(X^{\log}\) has a chart at \(\bar{x}\), and the stalk \(\omcM_{X,\bar{x}}\) of \(\omcM_X\) at \(\bar{x}\to X\) is finitely generated.
    \item
    Assume that the log structure of \(X^{\log}\) is saturated.
    Then we shall say that \(X^{\log}\) is an \textbf{fs} log scheme if for any geometric point \(\bar{x} \to X\), \(X^{\log}\) has a chart at \(\bar{x}\), and the stalk \(\omcM_{X,\bar{x}}\) of \(\omcM_X\) at \(\bar{x}\to X\) is finitely generated.
  \end{enumerate}
\end{defi}

\begin{rem}\label{rem: str bc is str}
  If \(f^{\log}:Y^{\log}\to X^{\log}\) is a strict morphism of log schemes and \(g^{\log}:Z^{\log}\to X^{\log}\) is a morphism of log schemes, then one verifies immediately that the natural morphism \(Y^{\log}\times_{X^{\log}}Z^{\log}\to Z^{\log}\) is strict.
\end{rem}

\begin{rem}\label{rem: sat log scheme}
  Assume that the log structure of a log scheme \(X^{\log}\) is integral.
  Then one verifies immediately that
  the log structure of \(X^{\log}\) is saturated if and only if
  for any geometric point \(\bar{x}\to X\),
  \(\omcM_{X,\bar{x}}\) is a saturated monoid.
\end{rem}

\begin{defi}\label{defi: cat pt}
  Let \(X^{\log}\) be an fs log scheme.
  \begin{enumerate}
    \item
    We shall say that
    \(X^{\log}\) is an \textbf{fs log point}
    if the underlying scheme \(X\) of \(X^{\log}\) is isomorphic to the spectrum of a field.
    \item
    Assume that \(X^{\log}\) is an fs log point.
    We shall say that \(X^{\log}\) is a \textbf{split log point}
    if \(\omcM_X\) is a constant sheaf on the \'etale site of \(X\). 
  \end{enumerate}
\end{defi}

\begin{defi}\label{defi: log res fld}
  Let \(f^{\log}: Y^{\log}\to X^{\log}\) be a morphism of fs log schemes.
  Assume that \(Y^{\log}\) is an fs log point.
  Write \(y\in Y\) for the unique point of \(Y\).
  Then we shall say that \(f^{\log}\) is a \textbf{log residue field} of \(X^{\log}\)
  if \(f^{\log}\) is strict, and
  \(f:Y\to X\) is isomorphic as an \(X\)-scheme to the natural morphism \(\Spec(k(f(y))) \to X\) that arises from the spectrum of the residue field at \(f(y)\in X\).
\end{defi}

\begin{lem}\label{lem: morph split log pt}
  Let \(X^{\log}\) be an fs log point and
  \(\bar{x}\to X\) a geometric point.
  Write \(M\dfn \omcM_{X,\bar{x}}\).
  Let \(\varphi: M\to N\) be a local morphism between sharp fs monoids.
  Then there exists a morphism between fs log points
  \(f^{\log}:Z^{\log}\to X^{\log}\)
  such that the following conditions hold:
  \begin{enumerate}
    \item \label{enumi: fin sep lem: morph split log pt}
    \(f^{\#}: \Gamma(X,\OX)\to \Gamma(Z,\OZ)\) is
    a finite separable field extension in \(k(\bar{x})\).
    \item \label{enumi: split lem: morph split log pt}
    \(Z^{\log}\) is a split log point.
    \item \label{enumi: given morph lem: morph split log pt}
    For any geometric point \(\bar{z}\to Z\),
    there exists an isomorphism \(\psi:N\xrightarrow{\sim} \omcM_{Z,\bar{z}}\)
    such that the following diagram commutes: 
    \[\begin{tikzcd}
      M \ar[r,"\varphi"] \arrow[equal]{d} &
      N \ar[d,"\psi","{\rotatebox{90}{\(\sim\)}}"'] \\
      M \ar[r,"{\bar{f}^{\flat}_{\bar{z}}}"]& \omcM_{Z,\bar{z}}.
    \end{tikzcd}\]
  \end{enumerate}
\end{lem}


\begin{proof}
  Write \(k\dfn \Gamma(X,\OX)\),
  \(k_s\subset k(\bar{x})\) for the subfield consisting of the elements which are separable algebraic over \(k\),
  \(G\) for the automorphism group of the \(k\)-algebra \(k_s\), and
  \[
  H\dfn \left\{ \sigma \in G \ \middle| \ \text{\(\sigma\) acts trivially on \(M\)}\right\}.
  \]
  Then, since \(M\) is finitely generated,
  the subgroup \(H\subset G\) is of finite index.
  Write \(K \subset k_s\) for the fixed field of \(H\) and \(f:Z\to X\) for the morphism of schemes determined by the field extension \(K/k\).
  Then, since \(H\subset G\) is of finite index, \(f^{\#}: \Gamma(X,\OX)\to \Gamma(Z,\OZ)\) satisfies condition \ref{enumi: fin sep lem: morph split log pt}.

  Write
  \begin{itemize}
    \item \(\varphi_Z: M_Z\to N_Z\) for the morphism between constant \'etale sheaves of monoids on the \'etale site of \(Z\) determined by the morphism of monoids \(\varphi:M\to N\);
    \item \(\Gamma(Z,f^{-1}\omcM_X)_Z\) for the constant \'etale sheaf of monoids on the \'etale site of \(Z\) determined by the monoid \(\Gamma(Z,f^{-1}\omcM_X)\) (which is naturally isomorphic to \(M\));
    \item \(Z^{\log}\) for the split log point such that \(Z=\Spec(K)\), and the log structure of \(Z^{\log}\) is the morphism of sheaves of monoids \[\alpha_Z: \mcM_Z\dfn \mcO_Z^{\times}\times N_Z\to \mcO_Z\] on the \'etale site of \(Z\) determined by the natural inclusion \(\mcO_Z^{\times} \hookrightarrow \mcO_Z\) and the unique local morphism of monoids \(N\to K\), where we regard \(K\) as a commutative monoid by the multiplication operation, and we recall that \(N\) is assumed to be sharp.
  \end{itemize}
  Then, since the log structure of \(Z^{\log}\) arises from the unique local morphism of monoids \(N\to K\), 
  \(Z^{\log}\) satisfies condition \ref{enumi: split lem: morph split log pt}.

  Since \(H\) acts trivially on \(M\), the natural morphisms \(f^{-1}\omcM_X \xleftarrow{\sim} \Gamma(Z,f^{-1}\omcM_X)_Z \xrightarrow{\sim} M_Z\) of sheaves of monoids on the \'etale site of \(Z\) are isomorphisms.
  Hence there exists an isomorphism \(f^{-1}\mcO_X^{\times} \times M_Z\xrightarrow{\sim} f^{-1}\mcM_X\).
  Since \(\varphi:M\to N\) is local, the morphism of sheaves of monoids \(f^{\flat}: f^{-1}\mcM_X\to \mcM_Z\) determined by \(f^{-1}\mcO_X^{\times} \times M_Z\xrightarrow{\sim} f^{-1}\mcM_X\) and \(f^{\#,\times}\times \varphi_Z: f^{-1}\mcO_X^{\times} \times M_Z\to \mcO_Z^{\times}\times N_Z\) satisfies \(\alpha_Z \circ f^{\flat} = f^{\#}\circ f^{-1}(\alpha_X)\).
  Thus the pair \((f,f^{\flat})\) is a morphism of log schemes \(f^{\log}: Z^{\log}\to X^{\log}\).
  Then \(f^{\log}\) satisfies condition \ref{enumi: given morph lem: morph split log pt}.
  This completes the proof of \autoref{lem: morph split log pt}.
\end{proof}

Next, we discuss some basic properties of push-outs
in the category of fs log schemes.

\begin{lem}\label{lem: small et exists}
  Let \(i:Y\hookrightarrow X\) be a closed immersion of schemes,
  \(\bar{y}\to Y\) a geometric point of \(Y\), and
  \(g:V\to Y\) an \'{e}tale neighborhood of \(\bar{y}\to Y\).
  Then there exist an \'{e}tale neighborhood \(f:U\to X\)
  of the geometric point \(\bar{y}\to Y\to X\) of \(X\) and
  a morphism \(h:U\times_XY\to V\) of \(Y\)-schemes such that
  the following diagram commutes and indeed is cartesian:
  \[\begin{tikzcd}
      U\times_X Y \ar[r,"h"]\ar[d, hookrightarrow] &
      V \ar[r,"g"] &
      Y\ar[d, hookrightarrow,"i"] \\
      U \ar[rr,"f"] & & X.
  \end{tikzcd}\]
\end{lem}

\begin{proof}
  Write \(\bar{x}\to X\) for the composite \(\bar{y}\to Y\to X\) and \(\mcI\dfn \ker(i^{\#}:\mathcal{O}_X\to i_*\mathcal{O}_Y)\).
  Then, by \cite[\href{https://stacks.math.columbia.edu/tag/05WS}{Tag 05WS}]{stacks-project},
  the morphism \(\mathcal{O}_{X,\bar{x}}/\mcI_{\bar{x}} \xrightarrow{\sim} \mcO_{Y,\bar{y}}\)
  induced by \(i^{\#}\) is an isomorphism.
  Thus \autoref{lem: small et exists} follows immediately from \cite[\href{https://stacks.math.columbia.edu/tag/04GW}{Tag 04GW}]{stacks-project}.
\end{proof}


\begin{lem}\label{coprod log sch}
  Let \(X^{\log} \xleftarrow{s^{\log}} Z^{\log} \xrightarrow{t^{\log}} Y^{\log}\)
  be morphisms of fs log schemes.
  Assume that the following conditions hold:
  \begin{itemize}
    \item \(s\) and \(t\) are closed immersions.
    \item Either \(s^{\log}\) or \(t^{\log}\) is strict. %
  \end{itemize}
  Write \(W\dfn X\amalg_ZY\) for the push-out of schemes
  (where we observe that the existence of \(W\) follows from
  \cite[\href{https://stacks.math.columbia.edu/tag/0E25}{Tag 0E25}]{stacks-project}),
  \(p:X\to W\) and \(q:Y\to W\) for the natural morphisms,
  \(r\dfn p\circ s = q\circ t\),
  \(\mcM_W \dfn p_*\mcM_X \times_{r_*\mcM_Z} q_*\mcM_Y\)
  (an \'{e}tale sheaf of monoids on \(W\)), and
  \(\alpha_W: \mcM_W\to p_*\mcO_X \times_{r_*\mcO_Z} q_*\mcO_Y \cong \mcO_W\)
  (cf. \cite[\href{https://stacks.math.columbia.edu/tag/0E25}{Tag 0E25}]{stacks-project}):
  \[
    \begin{tikzcd}
      Z \ar[r, hookrightarrow, "s"] \ar[rd, hookrightarrow, "r"] \ar[d, hookrightarrow, "t"] &
      X \ar[d, hookrightarrow, "p"] \\
      Y \ar[r, hookrightarrow, "q"] & W.
    \end{tikzcd}
  \]
  Then the following assertions hold:
  \begin{enumerate}
    \item \label{coprod log sch exists}
    The triple \(W^{\log} \dfn (W,\mcM_W, \alpha_W)\) is an fs log scheme.
    \item \label{coprod log sch push-out}
    \(W^{\log}\) represents the push-out of the diagram
    \(X^{\log} \xhookleftarrow{s^{\log}} Z^{\log} \xhookrightarrow{t^{\log}} Y^{\log}\)
    in the category of fs log schemes.
  \end{enumerate}
\end{lem}

\begin{proof}
  It follows immediately from the definition of \(W^{\log}\) that assertion \ref{coprod log sch push-out} follows from assertion \ref{coprod log sch exists}.
  In the remainder of the proof of \autoref{coprod log sch}, we prove assertion \ref{coprod log sch exists}.
  Since the functor \((-)^{\times}:\mathsf{Mon}\to \mathsf{Ab}\) preserves limits, it follows immediately that \(W^{\log}\) is a log scheme.
  Moreover, it follows immediately from \autoref{rem: int and sat monoids} \ref{enumi: sat rem: int and sat monoids} and the definition of the notion of a saturated log structure that the log structure of \(W^{\log}\) is saturated.
  Hence, to prove \autoref{coprod log sch},
  it suffices to prove that
  for any geometric point \(\bar{w}\to W\),
  \(W^{\log}\) has a chart \((U_0^{\log}\to W^{\log}, M_0 \to \Gamma(U_0,\mcO_{U_0}))\) at \(\bar{w}\) such that \(M_0\) is an fs monoid.
  Let \(\bar{w} \to W\) be a geometric point.

  Write \(U^{\log}\) for the fs log scheme
  determined by the open subset \(|X| \setminus \im(s)\),
  \(V^{\log}\) for the fs log scheme
  determined by the open subset \(|Y| \setminus \im(t)\), and
  \(f^{\log}: (U\sqcup V)^{\log}\to W^{\log}\)
  for the natural morphism of log schemes
  determined by \(p^{\log}\) and \(q^{\log}\).
  Then \(f^{\log}\) is a strict open immersion, and
  \(\im(f) = (\im(p)\setminus \im(q))\cup(\im(q)\setminus\im(p))\).
  Hence if \(\im(\bar{w}\to W) \subset (\im(p)\setminus \im(q))\cup(\im(q)\setminus\im(p))\),
  then \(W^{\log}\) has a chart at \(\bar{w}\) of the desired type.


  Assume that \(\im(\bar{w}\to W)\subset \im(r)\).
  Then the geometric point \(\bar{w}\to W\) arises from
  a geometric point \(\bar{w}\to Z\).
  Since \(s^{\log}\) and \(t^{\log}\) are morphisms of fs log schemes,
  it follows from \autoref{lem: small et exists} and
  \cite[Definition 2.9 (2) and Lemma 2.10]{Kato} that
  there exist a chart
  \((U_X^{\log}\to X^{\log}, M_X\to \Gamma(U_X, \OO{U_X}))\) at \(\bar{w}\to X\), a chart \((U_Y^{\log}\to Y^{\log}, M_Y\to \Gamma(U_Y, \OO{U_Y}))\) at \(\bar{w}\to Y\), a chart \((U_Z^{\log}\to Z^{\log}, M_Z\to \Gamma(U_Z, \OO{U_Z}))\) at \(\bar{w}\to Z\), a morphism of monoids \({\tilde{s}}^{\flat}: M_X\to M_Z\), and a morphism of monoids \({\tilde{t}}^{\flat}: M_Y\to M_Z\) such that \(M_X\), \(M_Y\), \(M_Z\) are fs monoids, and, moreover,
  for each \(u\in \{s,t\}\), if \(u^{\log}\) is strict,
  then \({\tilde{u}}^{\flat}\) is an isomorphism.
  By \autoref{lem: small et exists},
  there exists an \'{e}tale neighborhood \(U_W\to W\) of \(\bar{w}\to W\) such that
  for each \((*,?)\in \{(X,p),(Y,q),(Z,r)\}\),
  there exists a morphism of \(*\)-schemes \(U_W\times_W * \to U_*\) such that
  the following diagram commutes and indeed is cartesian:
  \[
    \begin{tikzcd}
      U_W\times_W * \ar[r, ""]\ar[d,hookrightarrow] &
      U_* \ar[r, "\text{\'{e}t}"] &
      * \ar[d, hookrightarrow, "?"] \\
      U_W \ar[rr, "\text{\'{e}t}"] & & W.
    \end{tikzcd}
  \]
  Write \(M_W\dfn M_X\times_{M_Z} M_Y\).
  Since either \({\tilde{s}}^{\flat}: M_X\to M_Z\) or
  \({\tilde{t}}^{\flat}: M_Y\to M_Z\) is an isomorphism,
  \(M_W\) is an fs monoid.
  Thus \(U_W\to W\) and the natural morphism
  \(M_W\to \Gamma(U_W,\OO{U_W})\) determine a chart of \(W^{\log}\) at \(\bar{w}\) of the desired type.
  This completes the proof of \autoref{coprod log sch}.
\end{proof}

\begin{cor}\label{cor: coprod exists}
  Let \(S^{\log}\) be an fs log scheme and \(X^{\log} \xhookleftarrow{s^{\log}} Z^{\log} \xhookrightarrow{t^{\log}} Y^{\log}\) strict closed immersions in \(\LSchb{S}\).
  Write \(W^{\log} \dfn X^{\log}\amalg_{Z^{\log}}Y^{\log}\) (cf. \autoref{coprod log sch}).
  Then the following assertions hold.
  \begin{enumerate}
    \item \label{enumi: W in LSchb cor: coprod exists}
    If \(\bbullet \subset \rqqs\), then
    \(W^{\log}\) belongs to the full subcategory \(\LSchb{S}\subset \LSch_{/S^{\log}}\),
    i.e., the push-out of \(s^{\log},t^{\log}\) exists in \(\LSchb{S}\).
    \item \label{enumi: ft cor: coprod exists}
    Assume that \(S^{\log}\) is locally Noetherian.
    If \(\bbullet \subset \{\mathrm{red,qcpt,qsep,sep,ft}\}\),
    then \(W^{\log}\) belongs to the full subcategory \(\LSchb{S}\subset \LSch_{/S^{\log}}\),
    i.e., the push-out of \(s^{\log},t^{\log}\) exists in \(\LSchb{S}\).
  \end{enumerate}
\end{cor}

\begin{proof}
  First, we note that
  by \autoref{coprod log sch} \ref{coprod log sch exists} \ref{coprod log sch push-out},
  it holds that \(W\cong X\amalg_Z Y\).
  Hence assertion \ref{enumi: W in LSchb cor: coprod exists} 
  follows immediately from \cite[\href{https://stacks.math.columbia.edu/tag/0E26}{Tag 0E26}]{stacks-project}, \cite[\href{https://stacks.math.columbia.edu/tag/04ZD}{Tag 04ZD}]{stacks-project}, and the fact that the fiber product of reduced rings is reduced.
  Assertion \ref{enumi: ft cor: coprod exists} follows immediately from assertion \ref{enumi: W in LSchb cor: coprod exists} and
  \cite[\href{https://stacks.math.columbia.edu/tag/0E27}{Tag 0E27}]{stacks-project}.
  This completes the proof of \autoref{cor: coprod exists}.
\end{proof}

\section{fs Log Points}
\label{section: Fs points}

In this section, we assume that \[\bbullet\subset \{\mathrm{red,qcpt,qsep,sep,ft}\}.\]

In the present \autoref{section: Fs points},
we give a category-theoretic characterization of the objects of \(\LSchb{S}\) whose underlying log scheme is an fs log point for locally Notherian fs log schemes \(S^{\log}\) (cf. \autoref{prop: fs log pt}).

First, we note the following lemma:

\begin{lem}\label{lem: empty and conn}
  Let \(S^{\log}\) be an fs log scheme and \(X^{\log}\) an object of \(\LSchb{S}\).
  Then the following assertions hold:
  \begin{enumerate}
    \item \label{enumi: empty lem: empty and conn}
    \(X\neq \emptyset\) if and only if
    \(X^{\log}\) is not an initial object of \(\LSchb{S}\).
    In particular,
    the property that \(X\neq \emptyset\)
    may be characterized category-theoretically
    from the data \((\LSchb{S},X^{\log})\).
    \item \label{enumi: conn lem: empty and conn}
    \(|X|\) is connected if and only if \(X\neq \emptyset\), and, moreover, \(X^{\log}\) does not admit a representation as a coproduct of two non-initial objects of \(\LSchb{S}\).
    In particular,
    the property that \(|X|\) is connected
    may be characterized category-theoretically
    from the data \((\LSchb{S},X^{\log})\).
  \end{enumerate}
\end{lem}

\begin{proof}
  Assertions \ref{enumi: empty lem: empty and conn}
  \ref{enumi: conn lem: empty and conn}
  follow immediately from elementary log scheme theory.
\end{proof}


Next, we give a category-theoretic characterization of fs log points
(cf. \autoref{prop: fs log pt}).

\begin{lem}\label{lem: epi from X X 1}
  Let \(S^{\log}\) be a locally Noetherian fs log scheme;
  \(i_1^{\log},i_2^{\log}:X^{\log} \to Y^{\log}\) morphisms in \(\LSchb{S}\);
  \(p^{\log}: Y^{\log}\to X^{\log}\) a morphism in \(\LSchb{S}\).
  Assume that the following conditions hold:
  \begin{enumerate}
    \item \label{enumi: X fs log pt epi from X X}
    \(X^{\log}\) is an fs log point.
    \item \label{enumi: id epi from X X}
    \(p^{\log}\circ i_1^{\log} = p^{\log}\circ i_2^{\log} = \id_{X^{\log}}\).
    \item \label{enumi: epi epi from X X}
    The morphism \(i^{\log}: X^{\log}\amalg X^{\log} \to Y^{\log}\) determined by \(i_1^{\log}\) and \(i_2^{\log}\) is an epimorphism in \(\LSchb{S}\).
    \item \label{enumi: connected epi from X X}
    \(|Y|\) is connected.
  \end{enumerate}
  Then \(i_1\), \(i_2\), \(p\) are isomorphisms, and \(i_1=i_2=p^{-1}\).
\end{lem}

\begin{proof}
  Write \(Z^{\log}\dfn \overline{\im(i)}_{\red}^{\log}\) (cf. \autoref{defi log sch} \ref{Zlog defi}), where we note that \(Z^{\log}\) belongs to the full subcategory \(\LSchb{S}\subset \LSch_{/S^{\log}}\).
  By condition \ref{enumi: X fs log pt epi from X X}, \(X\amalg X\) is reduced.
  Hence, by \cite[\href{https://stacks.math.columbia.edu/tag/056B}{Tag 056B}]{stacks-project}, the scheme-theoretic image of \(i\) is equal to \(Z\) (cf. \cite[\href{https://stacks.math.columbia.edu/tag/01R7}{Tag 01R7}]{stacks-project}, \cite[\href{https://stacks.math.columbia.edu/tag/01R6}{Tag 01R6}]{stacks-project}).
  Hence the morphism \(i^{\log}:X^{\log}\amalg X^{\log}\to Y^{\log}\) factors uniquely through the strict closed immersion \(Z^{\log} \hookrightarrow Y^{\log}\).
  By condition \ref{enumi: epi epi from X X}, \(Z^{\log}\hookrightarrow Y^{\log}\) is an epimorphism in \(\LSchb{S}\).
  Hence the two natural inclusions \(Y^{\log} \hookrightarrow Y^{\log}\amalg_{Z^{\log}} Y^{\log}\) coincide (where we note that by \autoref{cor: coprod exists} \ref{enumi: ft cor: coprod exists}, the push-out \(Y^{\log}\amalg_{Z^{\log}} Y^{\log}\) belongs to the full subcategory \(\LSchb{S}\subset \LSch_{/S^{\log}}\)).
  Thus, by the construction of \(Y^{\log} \hookrightarrow Y^{\log}\amalg_{Z^{\log}} Y^{\log}\) (cf. \autoref{coprod log sch} \ref{coprod log sch exists} and the proof of \cite[\href{https://stacks.math.columbia.edu/tag/0E25}{Tag 0E25}]{stacks-project}) and \autoref{cor: coprod exists} \ref{enumi: ft cor: coprod exists}, the strict closed immersion \(Z^{\log}\hookrightarrow Y^{\log}\) is an isomorphism.
  In particular, it holds that \(Y^{\log} = \overline{\im(i)}_{\red}^{\log}\).

  Write \(x\in X\) for the unique point (cf. condition \ref{enumi: X fs log pt epi from X X}).
  By condition \ref{enumi: id epi from X X}, the morphisms of fields \(k(x) \xrightarrow{} k(i_1(x))\) and \(k(x) \xrightarrow{} k(i_2(x))\) induced by \(p:Y\to X\) are isomorphisms.
  Hence, by \cite[\href{https://stacks.math.columbia.edu/tag/01TE}{Tag 01TE}]{stacks-project}, \(i_1(x)\) and \(i_2(x)\) are closed points of \(Y\).
  Thus it holds that
  \[|Y| = \overline{\im(i)} = \overline{\{i_1(x),i_2(x)\}} = \overline{\{i_1(x)\}} \cup \overline{\{i_2(x)\}} = \{i_1(x),i_2(x)\}.\]
  By condition \ref{enumi: connected epi from X X}, \(|Y|\) is of cardinality \(1\).
  Moreover, since \(Y = \overline{\im(i)}_{\red}\) is reduced, \(Y\) is isomorphic to the spectrum of a field.
  Thus, since \(p\circ i_1 = p\circ i_2 = \id_X\), the morphisms \(i_1\), \(i_2\), \(p\) are isomorphisms, and \(i_1=i_2=p^{-1}\).
  This completes the proof of \autoref{lem: epi from X X 1}.
\end{proof}

\begin{prop}\label{prop: fs log pt}
  Let \(S^{\log}\) be a locally Noetherian fs log scheme and \(X^{\log}\) an object of \(\LSchb{S}\).
  Assume that \(|X|\) is connected.
  Then \(X^{\log}\) is \textbf{not} an fs log point
  if and only if
  there exist an object \(Y^{\log}\in \LSchb{S}\),
  morphisms \(i_1^{\log},i_2^{\log}:X^{\log}\to Y^{\log}\) in \(\LSchb{S}\), and
  a morphism \(p^{\log}:Y^{\log}\to X^{\log}\) in \(\LSchb{S}\) such that the following conditions hold:
  \begin{enumerate}
    \item \label{enumi: id lem: fs}
    \(p^{\log}\circ i_1^{\log} = p^{\log}\circ i_2^{\log} = \id_{X^{\log}}\).
    \item \label{enumi: epi lem: fs}
    The morphism
    \(i^{\log} : X^{\log}\amalg X^{\log}\to Y^{\log}\)
    determined by \(i_1^{\log}\) and \(i_2^{\log}\) is an epimorphism in \(\LSchb{S}\).
    \item \label{enumi: conn lem: fs}
    \(|Y|\) is connected.
    \item \label{enumi: not isom lem: fs}
    Neither \(i_1^{\log}\) nor \(i_2^{\log}\) is an isomorphism.
    \item \label{enumi: univ surj lem: fs}
    For any morphism \(f^{\log}: Z^{\log}\to Y^{\log}\) in \(\LSchb{S}\) such that \(Z\neq \emptyset\), there exists a commutative diagram
    \[\begin{tikzcd}
      W^{\log} \ar[r] \ar[d]&
      Z^{\log} \ar[d,"f^{\log}"]& \\
      X^{\log}\amalg X^{\log} \ar[r,"i^{\log}"]& Y^{\log}
    \end{tikzcd}\]
    in \(\LSchb{S}\) such that \(W\neq \emptyset\).
  \end{enumerate}
  In particular,
  the property that \(X^{\log}\) is an fs log point
  may be characterized category-theoretically
  from the data \((\LSchb{S}, X^{\log})\)
  (cf. \autoref{lem: empty and conn} \ref{enumi: empty lem: empty and conn} \ref{enumi: conn lem: empty and conn}).
\end{prop}

\begin{proof}
  First, we prove necessity.
  Assume that \(X^{\log}\) is not an fs log point.
  Then \(X\) is not isomorphic to the spectrum of a field.
  Hence there exists a closed immersion \(j:X_1\hookrightarrow X\)
  such that \(X_1\) is reduced, \(X_1\neq \emptyset\), and \(j\) is not an isomorphism.
  The pull-back of the log structure of \(X^{\log}\) to \(X_1\) determines a log structure on \(X_1\), together with a strict closed immersion \(j^{\log}: X_1^{\log}\hookrightarrow X^{\log}\) in \(\LSchb{S}\).
  Write \(Y^{\log} \dfn X^{\log}\amalg_{X_1^{\log}}X^{\log}\) (cf. \autoref{cor: coprod exists} \ref{enumi: ft cor: coprod exists});
  \(i_1^{\log},i_2^{\log}:X^{\log}\to Y^{\log}\) for the natural inclusions;
  \(p^{\log}:Y^{\log}\to X^{\log}\) for the unique morphism such that \(p^{\log}\circ i_1^{\log} = p^{\log}\circ i_2^{\log} = \id_{X^{\log}}\).
  Then, by \autoref{cor: coprod exists} \ref{enumi: ft cor: coprod exists}, \(Y^{\log}\in \LSchb{S}\), and \((Y^{\log},i_1^{\log},i_2^{\log},p^{\log})\) satisfies conditions \ref{enumi: id lem: fs} and \ref{enumi: epi lem: fs}.
  Since \(|X|\) is connected, \(X_1\neq \emptyset\), and \(j:X_1\hookrightarrow X\) is not an isomorphism, it follows from the construction of \(X^{\log}\amalg_{X_1^{\log}}X^{\log}\) (cf. \autoref{coprod log sch} \ref{coprod log sch exists} and the proof of \cite[\href{https://stacks.math.columbia.edu/tag/0E25}{Tag 0E25}]{stacks-project}) that \((Y^{\log},i_1^{\log},i_2^{\log},p^{\log})\) satisfies conditions \ref{enumi: conn lem: fs} and \ref{enumi: not isom lem: fs}.
  Since \(j^{\log}:X_1^{\log}\hookrightarrow X^{\log}\) is a strict closed immersion, it follows from the construction of the log structure of \(X^{\log}\amalg_{X_1^{\log}}X^{\log}\) (cf. \autoref{coprod log sch} \ref{coprod log sch exists}) that the morphism \(i^{\log}: X^{\log}\amalg X^{\log} \to Y^{\log}\) determined by \(i_1^{\log}\) and \(i_2^{\log}\) is strict.
  Thus, since \(i\) is surjective, \((Y^{\log},i_1^{\log},i_2^{\log},p^{\log})\) satisfies condition \ref{enumi: univ surj lem: fs}.
  This completes the proof of necessity.


  Next, we prove sufficiency.
  Assume that \(X^{\log}\) is an fs log point.
  Let \(Y^{\log} \in \LSchb{S}\) be an object;
  \(i_1^{\log},i_2^{\log}:X^{\log}\to Y^{\log}\) morphisms in \(\LSchb{S}\);
  \(p^{\log}:Y^{\log}\to X^{\log}\) a morphism in \(\LSchb{S}\) such that \((Y^{\log}, i_1^{\log}, i_2^{\log}, p^{\log})\) satisfies conditions \ref{enumi: id lem: fs} \ref{enumi: epi lem: fs} \ref{enumi: conn lem: fs} \ref{enumi: not isom lem: fs}.
  To prove sufficiency, it suffices to prove that \((Y^{\log}, i_1^{\log}, i_2^{\log}, p^{\log})\) does not satisfy condition \ref{enumi: univ surj lem: fs}.
  Since \((Y^{\log}, i_1^{\log}, i_2^{\log}, p^{\log})\) satisfies conditions \ref{enumi: id lem: fs} \ref{enumi: epi lem: fs} \ref{enumi: conn lem: fs}, it follows from \autoref{lem: epi from X X 1}
  that \(i_1\), \(i_2\), \(p\) are isomorphisms, and \(i_1 = i_2 = p^{-1}\).
  Hence, to prove that \((Y^{\log}, i_1^{\log}, i_2^{\log}, p^{\log})\) does not satisfy condition \ref{enumi: univ surj lem: fs}, we may assume without loss of generality that \(X=Y\), and \(i_1=i_2=p=\id_X\).

  Let \(\bar{x}\to X\) be a geometric point.
  Write \(M\dfn \omcM_{X,\bar{x}}\), \(N\dfn \omcM_{Y,\bar{x}}\), \(\bar{i}_1^{\flat}\dfn \bar{i}^{\flat}_{1,\bar{x}}\), \(\bar{i}_2^{\flat}\dfn \bar{i}^{\flat}_{2,\bar{x}}\), and \(\bar{p}^{\flat}\dfn \bar{p}^{\flat}_{\bar{x}}\).
  Then \(M\), \(N\) are sharp fs monoids, and \(\bar{i}_1^{\flat}\), \(\bar{i}_2^{\flat}\), \(\bar{p}^{\flat}\) are local morphisms of monoids.
  Moreover, since \((Y^{\log}, i_1^{\log}, i_2^{\log}, p^{\log})\) satisfies condition \ref{enumi: id lem: fs}, it holds that \(\bar{i}_1^{\flat} \circ \bar{p}^{\flat} = \bar{i}_2^{\flat}\circ \bar{p}^{\flat} = \id_M\):
  \[\begin{tikzcd}
    M \ar[r,"\bar{p}^{\flat}"']\ar[rr,bend left=20,"\id_M"] &
    N \ar[r,shift left=0.3ex]
    \ar[r,shift right=0.3ex,swap,"{\bar{i}_1^{\flat},\bar{i}_2^{\flat}}"]&[1cm] M.
  \end{tikzcd}\]
  Since \((Y^{\log}, i_1^{\log}, i_2^{\log}, p^{\log})\) satisfies condition \ref{enumi: not isom lem: fs}, neither \(\bar{i}_1^{\flat}\) nor \(\bar{i}_2^{\flat}\) is injective.
  Hence, by \autoref{cor: monoid qint},
  there exists a sharp fs monoid \(N\subset L\subset N^{\gp}\)
  such that neither \(M\amalg_{\bar{i}_1^{\flat},N}L\) nor
  \(M\amalg_{\bar{i}_2^{\flat},N}L\) is quasi-integral.
  Since \(N\hookrightarrow L\) is local, it follows from \autoref{lem: morph split log pt} that there exists a morphism \(f^{\log}:Z^{\log}\to Y^{\log}\) in \(\LSchb{S}\) such that \(Z^{\log}\neq \emptyset\), and for any geometric point \(\bar{z}\to Z\),
  \(\bar{f}_{\bar{z}}^{\flat}: (f^{-1}\omcM_Y)_{\bar{z}} \to \omcM_{Z,\bar{z}}\) is isomorphic as an object of \(\Mon_{N/}\) to \(N\subset L\).
  Since neither \(M\amalg_{\bar{i}_1^{\flat},N}L\) nor
  \(M\amalg_{\bar{i}_2^{\flat},N}L\) is quasi-integral,
  it follows from \cite[Lemma 2.2.5]{Nak} that for each \(k\in \{1,2\}\),
  \(X^{\log}\times_{i_k^{\log},Y^{\log},f^{\log}}Z^{\log} = \emptyset\).
  Thus \((Y^{\log}, i_1^{\log}, i_2^{\log}, p^{\log})\) does not satisfy condition \ref{enumi: univ surj lem: fs}.
  This completes the proof of \autoref{prop: fs log pt}.
\end{proof}

Next, we prove the following property concerning fiber products in \(\LSchb{S}\).

\begin{lem}\label{lem: fiber prod}
  Let \(S^{\log}\) be a locally Noetherian fs log scheme, \(f^{\log}: X^{\log}\to Y^{\log}\) a morphism in \(\LSchb{S}\), and \(g^{\log}: Y_0^{\log}\to Y^{\log}\) a quasi-compact morphism in \(\LSchb{S}\).
  Then the fiber product \(X^{\log}\btimes_{Y^{\log}}Y_0^{\log}\) exists in \(\LSchb{S}\).
  Moreover, the natural morphism
  \[(X^{\log}\btimes_{Y^{\log}}Y_0^{\log})_{\red}\xrightarrow{\sim} (X^{\log}\times_{Y^{\log}}Y_0^{\log})_{\red}\]
  induced by the morphism of log schemes \(X^{\log}\btimes_{Y^{\log}}Y_0^{\log}\to X^{\log}\times_{Y^{\log}}Y_0^{\log}\) is an isomorphism of log schemes.
  In particular, if \(f^{\log}\) is strict, then the natural projection \(X^{\log}\btimes_{Y^{\log}}Y_0^{\log}\to Y_0^{\log}\) is also strict.
\end{lem}

\begin{proof}
  Write \(\bbbullet\dfn \bbullet\setminus\{\mathrm{red}\}\).
  By \autoref{rem: int and sat monoids} \ref{enumi: fg saturation is fg}, the underlying scheme of \(X^{\log}\times_{Y^{\log}}Y_0^{\log}\) is finite over \(X\times_Y Y_0\).
  Since \(g^{\log}: Y_0^{\log}\to Y^{\log}\) is quasi-compact, the natural projection \(X\times_Y Y_0\to X\) is quasi-compact.
  Thus \(X^{\log}\times_{Y^{\log}}Y_0^{\log}\) belongs to the full subcategory \(\LSchc{S}\subset \LSch_{/S^{\log}}\).
  If, moreover, \(\{\mathrm{red}\}\subset \bbullet\), then \((X^{\log}\times_{Y^{\log}}Y_0^{\log})_{\red}\) belongs to the full subcategory \(\LSchb{S}\subset \LSch_{/S^{\log}}\) and indeed may be interpreted as the fiber product \(X^{\log}\btimes_{Y^{\log}}Y_0^{\log}\) in \(\LSchb{S}\).
  In particular, for arbitrary \(\bbullet\), the natural morphism
  \[(X^{\log}\btimes_{Y^{\log}}Y_0^{\log})_{\red}\xrightarrow{\sim} (X^{\log}\times_{Y^{\log}}Y_0^{\log})_{\red}\]
  is an isomorphism.
  The final assertion follows immediately from \autoref{rem: str bc is str}.
  This completes the proof of \autoref{lem: fiber prod}.
\end{proof}

Finally, we give a category-theoretic characterization of log residue fields in \(\LSchb{S}\) (cf. \autoref{defi: log res fld}).

\begin{cor}\label{cor: log res fld}
  Let \(S^{\log}\) be a locally Noetherian fs log scheme and \(f^{\log} : Y^{\log} \to X^{\log}\) a morphism in \(\LSchb{S}\).
  Assume that \(Y^{\log}\) is an fs log point.
  Then \(f^{\log}\) is a log residue field of \(X^{\log}\) if and only if \(f^{\log}\) satisfies the following condition:
  \begin{enumerate}[start=2,label=(\fnsymbol*)]
    \item \label{enumi cor: log res fld condi}
    For any morphism \(g^{\log}: Z^{\log}\to X^{\log}\) in \(\LSchb{S}\) such that \(Z^{\log}\) is an fs log point, if \(Z^{\log}\btimes_{X^{\log}}Y^{\log}\neq \emptyset\), then the natural projection \(Z^{\log}\btimes_{X^{\log}}Y^{\log} \xrightarrow{\sim} Z^{\log}\) is an isomorphism (where we note that since \(Y^{\log}\) is an fs log point, the fiber product \(Z^{\log}\btimes_{X^{\log}}Y^{\log}\) exists in \(\LSchb{S}\), cf. \autoref{lem: fiber prod}):
    \[\begin{tikzcd}
      Z^{\log}\btimes_{X^{\log}}Y^{\log}\neq \emptyset \ar[r] \ar[d, "\rotatebox{90}{\(\sim\)}"'] &[1.5cm]
      Y^{\log} \ar[d,"f^{\log}"] \\
      Z^{\log} \ar[r,"g^{\log}"] 
      & X^{\log}
    \end{tikzcd}\]
  \end{enumerate}
  In particular, the property that
  \(f^{\log}: Y^{\log}\to X^{\log}\) is a log residue field of \(X^{\log}\)
  may be characterized category-theoretically
  (cf. \autoref{lem: empty and conn} \ref{enumi: empty lem: empty and conn}, \autoref{prop: fs log pt})
  from the data \((\LSchb{S}, f^{\log})\).
\end{cor}

\begin{proof}
  Necessity follows immediately from \autoref{lem: fiber prod}.
  In the remainder of the proof of \autoref{cor: log res fld}, we prove sufficiency.
  Assume that \(f^{\log}: Y^{\log}\to X^{\log}\) satisfies condition \ref{enumi cor: log res fld condi}.
  Write \(x\in \im(f)\) for the unique point contained in \(\im(f)\subset X\) and \(g^{\log}: Z^{\log} \to X^{\log}\) for the log residue field determined by \(x\in X\).
  Then, by the necessity portion of \autoref{cor: log res fld},
  the natural projection \(Z^{\log}\btimes_{X^{\log}}Y^{\log}\xrightarrow{\sim} Y^{\log}\neq \emptyset\) is an isomorphism.
  Hence, by condition \ref{enumi cor: log res fld condi}, the natural projection \(Z^{\log}\btimes_{X^{\log}}Y^{\log}\xrightarrow{\sim} Z^{\log}\) is an isomorphism.
  Thus \(f^{\log}: Y^{\log}\to X^{\log}\) is isomorphic as a log scheme over \(X^{\log}\) to the log residue field \(g^{\log}: Z^{\log} \to X^{\log}\).
  This completes the proof of \autoref{cor: log res fld}.
\end{proof}

\section{Strict Morphisms}
\label{section: str mor}

In this section, we assume that \[\bbullet\subset \{\mathrm{red,qcpt,qsep,sep,ft}\}.\]

In the present section, we give a category-theoretic characterization of strict morphisms in \(\LSchb{S}\) (cf. \autoref{strict general}).

First, we prove the following property of the \textit{strict locus} of a morphism of fs log schemes.

\begin{lem}\label{lem: str locus is qc open}
  Let \(f^{\log}:X^{\log}\to Y^{\log}\) be a morphism of fs log schemes.
  Write
  \[\mathrm{Str}(f^{\log}) \dfn \left\{x\in X \,\middle|\, \begin{array}{l}
    \text{the composite of \(f^{\log}: X^{\log} \to Y^{\log}\) with the log residue field} \\
    \text{\(\Spec(k(x))^{\log}\to X^{\log}\) determined by \(x\in X\) is strict}
  \end{array}\right\}. \]
  Then the following assertions hold:
  \begin{enumerate}
    \item \label{enumi: str locus open}
    \(\mathrm{Str}(f^{\log})\subset |X|\) is an open subset.
    Thus, \(\mathrm{Str}(f^{\log})\subset X\) may be regarded as an open subscheme.
    \item \label{enumi: str locus pb}
    Let
    \[\begin{tikzcd}
      X^{\log} \ar[r,"p^{\log}","\text{\rm strict}"'] \ar[d,"f^{\log}"'] &[1cm]
      X_0^{\log} \ar[d,"f_0^{\log}"] \\
      Y^{\log} \ar[r,"q^{\log}","\text{\rm strict}"'] &
      Y_0^{\log}
    \end{tikzcd}\]
    be a commutative diagram of fs log schemes such that \(p^{\log}\) and \(q^{\log}\) are strict.
    Then it holds that \(\mathrm{Str}(f^{\log}) = p^{-1}(\mathrm{Str}(f_0^{\log}))\).
    \item \label{enumi: str locus qc}
    The open immersion \(\mathrm{Str}(f^{\log})\hookrightarrow X\) (cf. \ref{enumi: str locus open}) is quasi-compact.
  \end{enumerate}
\end{lem}

\begin{proof}
  First, we prove assertion \ref{enumi: str locus open}.
  Let \(\bar{x}_0\to X\) be a geometric point such that \(\im(\bar{x}_0\to X)\subset \mathrm{Str}(f^{\log})\).
  Then \(\bar{f}^{\flat}_{\bar{x}_0}: (f^{-1}\omcM_Y)_{\bar{x}_0} \xrightarrow{\sim} \omcM_{X,\bar{x}_0}\) is an isomorphism.
  Hence, there exists an \'etale neighborhood \(i_0: U_0\to X\) of \(\bar{x}_0\) such that \(i_0^{-1}f^{-1}\omcM_Y \xrightarrow{\sim} i_0^{-1}\omcM_X\) is an isomorphism.
  Thus \(\im(\bar{x}_0\to X)\subset |\im(i_0)|\subset \mathrm{Str}(f^{\log})\).
  Since \(i_0\) is \'etale, \(|\im(i_0)|\subset |X|\) is an open subset.
  This implies that \(\mathrm{Str}(f^{\log})\subset |X|\) is an open subset.
  This completes the proof of assertion \ref{enumi: str locus open}.

  Next, we prove assertion \ref{enumi: str locus pb}.
  Let \(h^{\log}: Z^{\log}\to X^{\log}\) be a strict morphism of fs log schemes.
  Since \(p^{\log}\) is strict, \(p^{\log}\circ h^{\log}\) is also strict.
  Since \(q^{\log} \circ f^{\log} \circ h^{\log} = f_0^{\log}\circ p^{\log}\circ h^{\log}\), and \(q^{\log}\) is strict, it holds that
  \begin{quote}
    \(f^{\log}\circ h^{\log}\) is strict if and only if \(f_0^{\log}\circ p^{\log}\circ h^{\log}\) is strict.
  \end{quote}
  This implies that \(\mathrm{Str}(f^{\log}) = p^{-1}(\mathrm{Str}(f_0^{\log}))\).
  This completes the proof of assertion \ref{enumi: str locus pb}.

  Next, we prove assertion \ref{enumi: str locus qc}.
  By \cite[\href{https://stacks.math.columbia.edu/tag/02KQ}{Tag 02KQ}]{stacks-project} and \cite[\href{https://stacks.math.columbia.edu/tag/022C}{Tag 022C}]{stacks-project}, to prove assertion \ref{enumi: str locus qc}, it suffices to prove that there exists an \'etale covering \(\{f_i:U_i\to X\}_{i\in I}\) of \(X\) such that for any \(i\in I\), the open immersion \(f_i^{-1}(\mathrm{Str}(f^{\log}))\hookrightarrow U_i\) is a quasi-compact morphism.
  Hence, by assertion \ref{enumi: str locus pb} and \cite[Definition 2.9 (2), Lemma 2.10]{Kato}, to prove assertion \ref{enumi: str locus qc}, we may assume without loss of generality that there exist a morphism of fs monoids \(\psi:N\to M\) and a commutative diagram of fs log schemes
  \[\begin{tikzcd}
    X^{\log} \ar[d,"f^{\log}"'] \ar[r,"p^{\log}","\text{strict}"'] &[1cm]
    \Spec(\Z[M])^{\log} \ar[d,"\psi^{\log}"] \\
    Y^{\log} \ar[r,"q^{\log}","\text{strict}"'] &[1cm]
    \Spec(\Z[N])^{\log}
  \end{tikzcd}\]
  --- where we write \(\Spec(\Z[-])^{\log}\) for the log scheme determined by the monoid ring \(\Z[-]\) and the morphism of monoids \((-)\to \Z[-]\), and
  we write \(\psi^{\log}:\Spec(\Z[M])^{\log}\to \Spec(\Z[N])^{\log}\) for the morphism of log schemes determined by \(\psi: N\to M\) ---
  such that \(p^{\log}\) and \(q^{\log}\) are strict.
  Then, by assertion \ref{enumi: str locus pb}, it holds that \(\mathrm{Str}(f^{\log}) = p^{-1}(\mathrm{Str}(\psi^{\log}))\).
  Moreover, since \(\Spec(\Z[M])\) is Noetherian, the open immersion \(\mathrm{Str}(\psi^{\log}) \hookrightarrow \Spec(\Z[M])\) is a quasi-compact morphism (cf. \cite[\href{https://stacks.math.columbia.edu/tag/01OX}{Tag 01OX}]{stacks-project}).
  Thus \(\mathrm{Str}(f^{\log}) \hookrightarrow X\) is also a quasi-compact morphism (cf. \cite[\href{https://stacks.math.columbia.edu/tag/01K5}{Tag 01K5}]{stacks-project}).
  This completes the proof of \autoref{lem: str locus is qc open}.
\end{proof}

Next, we give a category-theoretic characterization of strict morphisms (cf. \autoref{strict general}).

\begin{lem}\label{exact mono fs pts}
  Let \(S^{\log}\) be a locally Noetherian fs log scheme and \(f^{\log}: X^{\log} \to Y^{\log}\) a morphism between fs log points in \(\LSchb{S}\).
  Assume that the diagonal morphism \(X^{\log} \to X^{\log}\btimes_{Y^{\log}} X^{\log}\) in \(\LSchb{S}\) is strict (where we note that since \(X^{\log}\) and \(Y^{\log}\) are fs log points, the fiber product \(X^{\log}\btimes_{Y^{\log}} X^{\log}\) exists in \(\LSchb{S}\), by \autoref{lem: fiber prod}).
  Then, for any geometric point \(\bar{x}\to X\), it holds that \[\coker(\bar{f}_{\bar{x}}^{\flat,\gp}: (f^{-1}\omcM_Y)_{\bar{x}}^{\gp} \to \omcM_{X,\bar{x}}^{\gp})\otimes_{\Z}\Q = 0.\]
\end{lem}

\begin{proof}
  Let \(\bar{x} \to X\) be a geometric point.
  Write
  \(M\dfn \omcM_{X,\bar{x}}\amalg_{\bar{f}_{\bar{x}}^{\flat},(f^{-1}\omcM_Y)_{\bar{x}},\bar{f}_{\bar{x}}^{\flat}}\omcM_{X,\bar{x}}\).
  Since the groupification functor \((-)^{\gp}: \Mon\to \mathsf{Ab}\) and the functor \((-)\otimes_{\Z}\Q\) preserve push-outs, it holds that
  \[\dim_{\Q}(M^{\gp}\otimes_{\Z}\Q) = 2\dim_{\Q}(\omcM_{X,\bar{x}}^{\gp}\otimes_{\Z}\Q) - \dim_{\Q}(\im(\bar{f}^{\flat,\gp}_{\bar{x}})\otimes_{\Z}\Q).\]
  Hence, to prove \autoref{exact mono fs pts}, it suffices to prove that
  \[\dim_{\Q}(M^{\gp}\otimes_{\Z}\Q) = \dim_{\Q}(\omcM_{X,\bar{x}}^{\gp}\otimes_{\Z}\Q).\]
  By (the proof of) \autoref{lem: fiber prod}, the diagonal morphism \(\Delta^{\log}:X^{\log}\to X^{\log}\times_{Y^{\log}}X^{\log}\) in \(\LSch_{/S^{\log}}\) is strict.
  Hence the composite of the natural morphisms
  \[\left(\mcM_{X,\bar{x}}\amalg_{f_{\bar{x}}^{\flat},(f^{-1}\mcM_Y)_{\bar{x}},f_{\bar{x}}^{\flat}}\mcM_{X,\bar{x}}\right)^{\sat}
  \xrightarrow{\sim} \left(\mcM_X\amalg_{f^{\flat},f^{-1}\mcM_Y,f^{\flat}}\mcM_X\right)_{\bar{x}}^{\sat} \to
  \left(\Delta^{-1}\mcM_{X^{\log}\times_{Y^{\log}}X^{\log}}\right)_{\bar{x}}\]
  induces an isomorphism
  \(
    \overline{\Delta}_{\bar{x}}^{\flat}: M^{\sat}/(M^{\sat})^{\times}\xrightarrow{\sim} \omcM_{X,\bar{x}}
  \).
  By \autoref{lem: qint Nak in other words}, \autoref{lem: quot is surj}, and \autoref{lem: sharp fs pushout sharp}, \(M\) is sharp, quasi-integral, and finitely generated.
  Hence \(M^{\inte}\) is sharp and fine.
  This implies that \(M^{\sat}\) is fs (cf. \autoref{rem: int and sat monoids} \ref{enumi: fg saturation is fg}), and, moreover, \((M^{\sat})^{\times}\) is a finite abelian group.
  Hence the morphism \(M^{\gp}\otimes_{\Z}\Q \xrightarrow{\sim} \omcM_{X,\bar{x}}^{\gp}\otimes_{\Z}\Q\) induced by \(\overline{\Delta}_{\bar{x}}^{\flat}\) is an isomorphism.
  In particular, it holds that
  \(\dim_{\Q}(M^{\gp}\otimes_{\Z}\Q) = \dim_{\Q}(\omcM_{X,\bar{x}}^{\gp}\otimes_{\Z}\Q)\).
  This completes the proof of \autoref{exact mono fs pts}.
\end{proof}

\begin{lem}\label{lem: Kummer type irred strict}
  Let \(S^{\log}\) be a locally Noetherian fs log scheme and \(p^{\log}: X^{\log}\to Y^{\log}\) a morphism in \(\LSchb{S}\).
  Assume that the following conditions hold:
  \begin{enumerate}
    \item
    \(|X|\) is irreducible, and \(Y^{\log}\) is an fs log point.
    Write \(\eta\in X\) for the unique generic point.
    \item \label{enumi: 1 Kummer type strict}
    For any geometric point \(\bar{x}\to X\) such that the log residue field determined by the image of \(\bar{x}\to X\) belongs to the full subcategory \(\LSchb{S}\subset \LSch_{/S^{\log}}\), it holds that \(\coker(\bar{p}_{\bar{x}}^{\flat,\gp})\otimes_{\Z}\Q = 0\).
    \item \label{enumi: 2 Kummer type strict}
    The composite of \(p^{\log}: X^{\log}\to Y^{\log}\) and the log residue field \(\Spec(k(\eta))^{\log}\to X^{\log}\) determined by \(\eta\in X\) is strict.
  \end{enumerate}
  Then \(p^{\log}\) is strict.
\end{lem}


\begin{proof}
  Let \(\bar{x}\to X\) be a geometric point such that the log residue field \(\Spec(k(x))^{\log}\to X^{\log}\) determined by \(\bar{x}\to X\) belongs to the full subcategory \(\LSchb{S}\subset \LSch_{/S^{\log}}\) and \(\bar{\eta}\to X\) a geometric point such that \(\eta\in \im(\bar{\eta}\to X)\).
  Then, since \(Y^{\log}\) is an fs log point, the localization morphism \((p^{-1}\omcM_Y)_{\bar{x}} \xrightarrow{\sim} (p^{-1}\omcM_Y)_{\bar{\eta}}\) is an isomorphism.
  Moreover, by condition \ref{enumi: 2 Kummer type strict}, \(\bar{p}^{\flat}_{\bar{\eta}}:(p^{-1}\omcM_Y)_{\bar{\eta}}\xrightarrow{\sim} \omcM_{X,\bar{\eta}}\) is an isomorphism.
  Since the following diagram of monoids commutes, \(\coker(\bar{p}_{\bar{x}}^{\flat,\gp})\) is a direct summand of \(\omcM_{X,\bar{x}}^{\gp}\):
  \[\begin{tikzcd}
    (p^{-1}\omcM_Y)_{\bar{x}} \ar[r,"\sim"] \ar[d,"\bar{p}^{\flat}_{\bar{x}}"'] &[1.5cm] (p^{-1}\omcM_Y)_{\bar{\eta}} \ar[d,"\bar{p}^{\flat}_{\bar{\eta}}","{\rotatebox{90}{\(\sim\)}}"'] \\
    \omcM_{X,\bar{x}} \ar[r,"\textrm{localization}"] & \omcM_{X,\bar{\eta}},
  \end{tikzcd}\]
  where \(\omcM_{X,\bar{x}} \to \omcM_{X,\bar{\eta}}\) is the localization morphism.
  Thus, since \(\omcM_{X,\bar{x}}^{\gp}\) is torsion-free, \(\coker(\bar{p}_{\bar{x}}^{\flat,\gp})\) is also torsion-free.
  Hence, by condition \ref{enumi: 1 Kummer type strict}, \(\bar{p}_{\bar{x}}^{\flat,\gp}\) is surjective.
  Since the above diagram commutes, \(\bar{p}_{\bar{x}}^{\flat,\gp}\) is an isomorphism, and, moreover, \(\omcM_{X,\bar{x}} \to \omcM_{X,\bar{\eta}}\) is injective.
  Hence, \(\bar{p}_{\bar{x}}^{\flat}\) is also an isomorphism.
  Since \(\omcM_{X,\bar{x}}\) and \((p^{-1}\omcM_Y)_{\bar{x}}\) are fs monoids,
  this implies that there exists an open subscheme \(U\subset X\) such that the following conditions hold:
  \begin{itemize}
    \item \(\bar{p}^{\flat}|_U: (p^{-1}\omcM_Y)|_U \to \omcM_X|_U\) is an isomorphism.
    \item For any geometric morphism \(\bar{x}\to X\), if the log residue field \(\Spec(k(x))^{\log}\to X^{\log}\) determined by \(\bar{x}\to X\) belongs to the full subcategory \(\LSchb{S}\subset \LSch_{/S^{\log}}\), then \(\im(\bar{x}\to X)\subset U\).
  \end{itemize}
  Since \(S\) is locally Noetherian, the second condition implies that \(U=X\).
  Thus \(\bar{p}^{\flat}\) is an isomorphism.
  This completes the proof of \autoref{lem: Kummer type irred strict}.
\end{proof}

\begin{prop}\label{ex grp obj over fld}
  Let \(S^{\log}\) be a locally Noetherian fs log scheme and \(p^{\log}: G^{\log} \to Y^{\log}\) a group object in \(\LSchb{S}\) over \(Y^{\log}\) such that \(Y^{\log}\) is an fs log point.
  Write \(e^{\log}: Y^{\log}\to G^{\log}\) for the identity section.
  Assume that
  the following conditions hold:
  \begin{enumerate}[label=(\alph*)]
    \item \label{log grp obj condi 0}
    \(|G|\) is connected.
    \item \label{log grp obj condi 1}
    The identity section \(e^{\log}: Y^{\log}\to G^{\log}\) is a log residue field.
  \end{enumerate}
  Then the following assertions hold:
  \begin{enumerate}
    \item \label{grp obj of Kummer type}
    For any geometric point \(\bar{g}\to G\) such that the log residue field determined by the image of \(\bar{g}\to G\) belongs to the full subcategory \(\LSchb{S}\subset \LSch_{/S^{\log}}\), it holds that \(\coker(\bar{p}_{\bar{g}}^{\flat,\gp})\otimes_{\Z}\Q = 0\).
    \item \label{grp obj geom irred and strict}
    \(p^{\log}\) is strict.
  \end{enumerate}
\end{prop}

\begin{proof}
  First, we prove assertion \ref{grp obj of Kummer type}.
  Let \(\bar{g}\to G\) be a geometric point such that the log residue field determined by the image of \(\bar{g}\to G\) belongs to the full subcategory \(\LSchb{S}\subset \LSch_{/S^{\log}}\).
  Write
  \begin{itemize}
    \item \(g^{\log}: Y_0^{\log}\to G^{\log}\) for the log residue field determined by the image of \(\bar{g}\to G\);
    \item
    \(G_0^{\log}\dfn G^{\log}\btimes_{Y^{\log},p^{\log}\circ g^{\log}} Y_0^{\log}\to Y_0^{\log}\) for the group object in \(\LSchb{S}\) obtained by base-changing \(p^{\log}\) by \(p^{\log}\circ g^{\log}\) (where we note that since \(p^{\log}\circ g^{\log}: Y_0^{\log}\to Y^{\log}\) is a morphism between fs log points, the fiber product \(G^{\log}\btimes_{Y^{\log},p^{\log}\circ g^{\log}} Y_0^{\log}\) exists in \(\LSchb{S}\), cf. \autoref{lem: fiber prod});
    \item
    \(e_0^{\log}\dfn e^{\log}\btimes_{Y^{\log},p^{\log}\circ g^{\log}}\id_{Y_0^{\log}}: Y_0^{\log}\to G_0^{\log}\) for the identity section of \(G_0^{\log}\);
    \item
    \(g_0^{\log}: Y_0^{\log}\to G_0^{\log}\) for the morphism over \(Y_0^{\log}\) determined by \(g^{\log}: Y_0^{\log}\to G^{\log}\);
    \item
    \(\varphi^{\log}:G_0^{\log}\xrightarrow{\sim} G_0^{\log}\) for the left translation isomorphism of \(G_0^{\log}\) determined by the morphism \(g_0^{\log}: Y_0^{\log}\to G_0^{\log}\) over \(Y_0^{\log}\).
  \end{itemize}
  Then \(g_0^{\log} = \varphi^{\log}\circ e_0^{\log}\).
  Moreover, by condition \ref{log grp obj condi 1} and \autoref{lem: fiber prod}, \(e_0^{\log}\) is strict.
  In particular, \(g_0^{\log}\) is strict.
  Since \(g^{\log}\) is strict, it follows from \autoref{lem: fiber prod} that the natural projection \(Y_0^{\log}\btimes_{Y^{\log}} Y_0^{\log} \to G_0^{\log}\) (cf. the commutative diagram below) is also strict.
  Thus the diagonal morphism \(Y_0^{\log}\to Y_0^{\log}\btimes_{Y^{\log}}Y_0^{\log}\) is strict:
  \[\begin{tikzcd}
    Y_0^{\log} \ar[r]\ar[rr,bend left=15,"g_0^{\log}\textrm{ :strict}"] &[0.4cm]
    Y_0^{\log}\btimes_{Y^{\log}}Y_0^{\log} \ar[r,"\textrm{strict}"'] \ar[d] &[0.5cm]
    G_0^{\log} \ar[r] \ar[d] &
    Y_0^{\log} \ar[d,"p^{\log}\circ g^{\log}"] \\
    & Y_0^{\log} \ar[r,"g^{\log}","\textrm{log residue field}"'] \ar[ru,bend left=1,phantom,"\ \ {\scriptstyle\square}"]&
    G^{\log} \ar[r,"p^{\log}"] \ar[ru,phantom,bend left=2,"{\scriptstyle \ \square}"]& Y^{\log}.
  \end{tikzcd}\]
  Hence it follows from \autoref{exact mono fs pts} that \(\coker(\bar{p}_{\bar{g}}^{\flat,\gp})\otimes_{\Z}\Q = 0\).
  This completes the proof of assertion \ref{grp obj of Kummer type}.

  Next, we prove assertion \ref{grp obj geom irred and strict}.
  By condition \ref{log grp obj condi 1}, \(\mathrm{Str}(p^{\log})\neq \emptyset\) (cf. \autoref{lem: str locus is qc open}).
  Let \(\eta \in \mathrm{Str}(p^{\log})\) be a point.
  Then, by assertion \ref{grp obj of Kummer type} and \autoref{lem: Kummer type irred strict}, the composite \[\overline{\{\eta\}}_{\red}^{\log}\hookrightarrow G^{\log} \xrightarrow{p^{\log}} Y^{\log}\] (cf. \autoref{defi log sch} \ref{Zlog defi}) is strict (where we note that since \(G^{\log}\in \LSchb{S}\), and \(\overline{\{\eta\}}_{\red}^{\log} \hookrightarrow G^{\log}\) is a strict closed immersion, \(\overline{\{\eta\}}_{\red}^{\log}\) belongs to the full subcategory \(\LSchb{S}\subset \LSch_{/S^{\log}}\)).
  Hence \(\overline{\{\eta\}}\subset \mathrm{Str}(p^{\log})\).
  In particular, we conclude that \(|\mathrm{Str}(p^{\log})| \subset |G|\) is stable under specialization.
  Moreover, by \autoref{lem: str locus is qc open} \ref{enumi: str locus open} \ref{enumi: str locus qc}, \(\mathrm{Str}(p^{\log})\hookrightarrow G\) is a quasi-compact open immersion.
  Thus, by \cite[\href{https://stacks.math.columbia.edu/tag/05JL}{Tag 05JL}]{stacks-project}, \(|\mathrm{Str}(p^{\log})|\subset |G|\) is closed.
  Since \(|G|\) is connected, we conclude that \(\mathrm{Str}(p^{\log}) = G\).
  This completes the proof of \autoref{ex grp obj over fld}.
\end{proof}

\begin{cor}\label{fs pt str}
  Let \(S^{\log}\) be a locally Noetherian fs log scheme and \(f^{\log}: X^{\log}\to Y^{\log}\) a morphism in \(\LSchb{S}\).
  Assume that \(X^{\log}\) and \(Y^{\log}\) are fs log points.
  Then \(f^{\log}\) is strict if and only if
  there exist a group object \(G^{\log}\to Y^{\log}\) in \(\LSchb{S}\) over \(Y^{\log}\) and
  a morphism \(g^{\log}: X^{\log} \to G^{\log}\) over \(Y^{\log}\) in \(\LSchb{S}\)
  such that the following conditions hold:
  \begin{enumerate}
    \item \label{fs pt str 0}
    \(|G|\) is connected.
    \item \label{fs pt str 1}
    The identity section \(e^{\log}:Y^{\log}\to G^{\log}\) is a log residue field.
    \item \label{fs pt str 3}
    \(g^{\log}\) is a log residue field.
  \end{enumerate}
  In particular,
  the property that
  \begin{itemize}
    \item[ \ ] \(f^{\log}\) is a strict morphism between fs log points
  \end{itemize}
  may be characterized category-theoretically (cf. \autoref{lem: empty and conn} \ref{enumi: conn lem: empty and conn}, \autoref{prop: fs log pt}, \autoref{cor: log res fld})
  from the data \((\LSchb{S},f^{\log})\).
\end{cor}

\begin{proof}
  First, we prove necessity.
  Assume that \(f^{\log}: X^{\log} \to Y^{\log}\) is a strict morphism between fs log points.
  Write
  \begin{itemize}
    \item \(k\dfn \Gamma(Y,\OY)\),
    \item \(A\) for the symmetric \(k\)-algebra determined by the underlying \(k\)-linear space of \(\Gamma(X,\OX)\),
    \item \(G\dfn \Spec(A) \to Y\) for the affine space over \(Y\) determined by the \(k\)-linear space \(\Gamma(X,\OX)\) equipped with the natural (additive) group scheme structure over \(Y\), and
    \item \(p^{\log}: G^{\log}\to Y^{\log}\) for the strict morphism obtained by pulling back the log structure of \(Y^{\log}\) via \(G\to Y\).
  \end{itemize}
  Note that since \(X^{\log}\) and \(Y^{\log}\) are fs log points, if \(\{\mathrm{ft}\}\subset \bbullet\), then the underlying morphism of schemes \(f:X\to Y\) is finite.
  Hence, regardless of whether or not \(\{\mathrm{ft}\}\subset \bbullet\), it holds that \(G^{\log}\in \LSchb{S}\).
  Since \(p^{\log}\) is strict, and \(G\) is a geometrically integral affine scheme over \(Y\), \(G^{\log}\times_{Y^{\log}}G^{\log}\) and \(G^{\log}\times_{Y^{\log}}G^{\log}\times_{Y^{\log}}G^{\log}\) belong to the full subcategory \(\LSchb{S}\subset \LSch_{/S^{\log}}\).
  Hence the evident group object structure of \(G^{\log}\) in \(\LSch_{/S^{\log}}\) over \(Y^{\log}\) may be regarded as a group object structure of \(G^{\log}\) in \(\LSchb{S}\) over \(Y^{\log}\).
  Moreover, the group object \(G^{\log}\) in \(\LSchb{S}\) over \(Y^{\log}\) clearly satisfies conditions \ref{fs pt str 0} and \ref{fs pt str 1}.

  Write \(g: X\to G\) for the closed immersion determined by the tautological surjection of \(k\)-algebras \(A\to \Gamma(X,\OX)\).
  Then \(f = p\circ g\).
  Since \(p^{\log}: G^{\log}\to Y^{\log}\) and \(f^{\log}:X^{\log} \to Y^{\log}\) are strict, the pull-back of the log structure of \(G^{\log}\) to \(X\) via \(g\) is isomorphic to the log structure of \(X^{\log}\).
  Thus we obtain a strict closed immersion \(g^{\log}: X^{\log}\to G^{\log}\) such that \(f^{\log} = p^{\log}\circ g^{\log}\).
  In particular, \(g^{\log}\) satisfies condition \ref{fs pt str 3}.
  This completes the proof of necessity.

  Next, we prove sufficiency.
  Assume that there exist a group object \(G^{\log}\to Y^{\log}\) in \(\LSchb{S}\) over \(Y^{\log}\) and
  a morphism \(g^{\log}: X^{\log} \to G^{\log}\) over \(Y^{\log}\) in \(\LSchb{S}\) such that \(G^{\log}\to Y^{\log}\) and \(g^{\log}\) satisfy conditions
  \ref{fs pt str 0}, \ref{fs pt str 1}, and \ref{fs pt str 3}.
  Since \(G^{\log}\to Y^{\log}\) satisfies conditions
  \ref{fs pt str 0} and \ref{fs pt str 1},
  it follows from \autoref{ex grp obj over fld} that \(G^{\log}\to Y^{\log}\) is strict.
  Thus, by condition \ref{fs pt str 3}, \(f^{\log}: X^{\log}\to Y^{\log}\) is also strict.
  This completes the proof of \autoref{fs pt str}.
\end{proof}

\begin{cor}\label{strict general}
  Let \(S^{\log}\) be a locally Noetherian fs log scheme and \(f^{\log}:X^{\log}\to Y^{\log}\) a morphism in \(\LSchb{S}\).
  Then \(f^{\log}\) is strict if and only if
  for any commutative diagram
  \[\begin{tikzcd}
    Z^{\log} \ar[r,"i^{\log}"]
    \ar[d,"p^{\log}"'] &[0.5cm]
    X^{\log} \ar[d,"f^{\log}"] \\
    W^{\log} \ar[r,"j^{\log}"] &
    Y^{\log}
  \end{tikzcd}\]
  in \(\LSchb{S}\), if
  \(i^{\log}\) is a log residue field of \(X^{\log}\), and
  \(j^{\log}\) is a log residue field of \(Y^{\log}\), then
  \(p^{\log}\) is strict.
  In particular,
  the property that \(f^{\log}\) is strict may be characterized category-theoretically
  (cf. \autoref{cor: log res fld}, \autoref{fs pt str})
  from the data \((\LSchb{S},f^{\log})\).
\end{cor}

\begin{proof}
  \autoref{strict general} follows immediately from \autoref{fs pt str}. 
\end{proof}

Now we prove the first equality of \autoref{corA bb bbb}.

\begin{cor}\label{bbul bbbul}
  Let \(S^{\log}, T^{\log}\) be locally Noetherian fs log schemes and \(\bbullet, \bbbullet \subset \{\mathrm{red, qcpt, qsep, sep}\}\) subsets such that \(\{\mathrm{qsep, sep}\}\not\subset \bbullet\), and \(\{\mathrm{qsep, sep}\}\not\subset \bbbullet\).
  If the categories \(\LSchb{S}\) and \(\LSchc{T}\) are equivalent, then \(\bbullet = \bbbullet\).
\end{cor}

\begin{proof}
  \autoref{bbul bbbul} follows immediately from
  \autoref{strict general}, \cite[Corollary 4.11]{YJ}, and \cite[\href{https://stacks.math.columbia.edu/tag/01OY}{Tag 01OY}]{stacks-project}, where we apply \cite[Corollary 4.11]{YJ} to the categories obtained by considering the full subcategory of \(\LSchb{S}\) or \(\LSchc{T}\) determined by the objects whose structure morphism is strict.
\end{proof}

\section{Log-like Morphisms}
\label{section: log like mor}

In this section, we assume that \[\bbullet\subset \{\mathrm{red,qcpt,qsep,sep,ft}\}.\]

In the present section, we give a category-theoretic characterization of the morphisms of monoid objects in \(\LSchb{S}\) that represent the functor
\begin{align*}
  \LSchb{S} &\to \Mor(\Mon) \\
  X^{\log} &\mapsto [\alpha_X(X): \Gamma(X,\mcM_X) \to \Gamma(X,\OX)],
\end{align*}
which arises from the log structures of objects of \(\LSchb{S}\) (cf. \autoref{rem: A^1log is a monoid obj}, \autoref{monoid object A}).
We then use this characterization to complete the proof of the main theorem of the present paper (cf. \autoref{thm: functorial reconstruction general}, \autoref{last cor}).

First, we introduce some notation used in this section.

\begin{defi}\label{5-1}
  \
  \begin{itemize}
    \item
    Let \(\mcC\) be a category. Then we shall write \(\Mon(\mcC)\) for the category of monoid objects in \(\mcC\).
    \item
    Let \(\mcC\) be a category and \(A\) a ring object in \(\mcC\).
    Then we shall write \((A,\times)\) for the underlying multiplicative monoid object of the ring object \(A\) in \(\mcC\).
    \item
    We shall write \(\A_{\Z}^{1,\log}\) for the fs log scheme over \(\Z\)
    whose underlying scheme is \(\A^1_{\Z} = \Spec(\Z[t])\), and
    whose log structure is the log structure determined by the morphism of monoids
    \begin{align*}
      \N &\to \Z[t], \\
      n &\mapsto t^n.
    \end{align*}
    \item
    We shall write \(\G_{m,\Z}\subset \A^1_{\Z}\) for the unit group scheme of the ring scheme \(\A^1_{\Z}\) over \(\Z\).
    \item
    Let \(X^{\log}\) be an fs log scheme and \(Y\) a scheme.
    Then we shall write \(X^{\log}\times_{\Z}Y\) for the fs log scheme whose underlying scheme is \(X\times_{\Z}Y\), and whose log structure is the log structure obtained by pulling back the log structure of \(X^{\log}\) via the natural projection \(X\times_{\Z}Y\to X\).
  \end{itemize}
\end{defi}

\begin{defi}[\(\alpha_{S^{\log},\A}^{\log}\)]
  \label{rem: A^1log is a monoid obj}
  Let \(S^{\log}\) be a locally Noetherian fs log scheme and \(f^{\log}: X^{\log}\to S^{\log}\) a morphism of fs log schemes.
  Then an element \(m\in \Gamma(X,\mcM_X)\) determines a morphism of monoids
  \(\bar{g}^{\flat}:\N\to \Gamma(X,\mcM_X)\) and
  a morphism of rings \(g^{\#}:\Z[t] \to \Gamma(X,\OX)\) such that the following diagram commutes:
  \[\begin{tikzcd}
    \N \ar[d,"1\mapsto t"'] \ar[r,"\bar{g}^{\flat}:1\mapsto m"]&[1.5cm]
    \Gamma(X,\mcM_X) \ar[d,"\alpha_X(X)"] \\
    \Z[t] \ar[r,"g^{\#}:t\mapsto \alpha_X(X)(m)"]& \Gamma(X,\OX).
  \end{tikzcd}\]
  Hence we obtain a morphism of fs log schemes \(g^{\log}: X^{\log}\to \A_{\Z}^{1,\log}\).
  Conversely, each morphism of fs log schemes \(g^{\log}:X^{\log}\to \A_{\Z}^{1,\log}\) determines an element \(g^{\flat}(1)\in \Gamma(X,\mcM_X)\).
  One verifies easily that these assignments determine an isomorphism of functors \(\Gamma(-,\mcM_{(-)})\xrightarrow{\sim} \Hom_{S^{\log}}(-,S^{\log}\times_{\Z}\A_{\Z}^{1,\log})\).
  In particular, \(S^{\log}\times_{\Z}\A_{\Z}^{1,\log}\) represents the functor
  \begin{align*}
    \LSch_{/S^{\log}} &\to \Mon, \\
    X^{\log} &\mapsto \Gamma(X,\mcM_X),
  \end{align*}
  which implies that \(S^{\log}\times_{\Z}\A_{\Z}^{1,\log}\) has a monoid object structure in \(\LSch_{/S^{\log}}\), hence also in \(\LSchb{S}\).
  Moreover, the family of morphisms of monoids \[\left\{\alpha_X(X):\Gamma(X,\mcM_X)\to \Gamma(X,\OX)\right\}_{X^{\log}\in \LSchb{S}}\]
  determines a morphism of monoid objects \(\alpha_{S^{\log},\A}^{\log}:S^{\log}\times_{\Z}\A_{\Z}^{1,\log} \to (S^{\log}\times_{\Z}\A^1_{\Z},\times)\) in \(\LSchb{S}\).
\end{defi}

The log structure of \(S^{\log}\times_{\Z}\A^{1,\log}_{\Z}\) may be identified with the log structure obtained by pushing forward the log structure of \(S^{\log}\times_{\Z}\G_{m,\Z}\) to \(S\times_{\Z}\A^1_{\Z}\) via the open immersion \(S\times_{\Z}\G_{m,\Z}\hookrightarrow S\times_{\Z}\A^1_{\Z}\) as follows (in the case where the log structure of \(S^{\log}\) is trivial, cf. \cite[Example 1.5 (1) (2)]{Kato}):

\begin{lem}\label{lem: A1log is push log str}
  Let \(S^{\log}\) be an fs log scheme.
  Then the log structure of \(S^{\log}\times_{\Z}\A_{\Z}^{1,\log}\) is isomorphic to the log structure determined by the push-forward of the log structure of \(S^{\log}\times_{\Z}\G_{m,\Z}\) to \(S\times_{\Z}\A_{\Z}^1\) via the open immersion \(S\times_{\Z}\G_{m,\Z}\hookrightarrow S\times_{\Z}\A^1_{\Z}\).
\end{lem}

\begin{proof}
  Write
  \begin{itemize}
    \item \(i:S\times_{\Z}\G_{m,\Z}\hookrightarrow S\times_{\Z}\A^1_{\Z}\) for the inclusion morphism,
    \item \(p:S\times_{\Z}\A^1_{\Z}\to S\) for the natural projection,
    \item \(\alpha_1:\mcM_1\to \mcO_{S\times_{\Z}\A^1_{\Z}}\) for the log structure of \(S^{\log}\times_{\Z}\A_{\Z}^{1,\log}\), and
    \item \(\alpha_2:\mcM_2\to \mcO_{S\times_{\Z}\A^1_{\Z}}\) for the log structure on \(S\times_{\Z}\A^1_{\Z}\) obtained by pushing forward the log structure of \(S^{\log}\times_{\Z}\G_{m,\Z}\) to \(S\times_{\Z}\A_{\Z}^1\) via the open immersion \(i:S\times_{\Z}\G_{m,\Z}\hookrightarrow S\times_{\Z}\A^1_{\Z}\).
  \end{itemize}
  Since the log structure of \(S^{\log}\times_{\Z}\G_{m,\Z}\) is isomorphic to the log structure \(i^*(\mcM_1,\alpha_1)\), there exists a unique morphism of sheaves of monoids \(\varphi: \mcM_1\to \mcM_2\) on the \'etale site of \(S\times_{\Z}\A^1_{\Z}\) such that the following diagram of sheaves of monoids on \(S\times_{\Z}\A^1_{\Z}\) commutes:
  \[\begin{tikzcd}
    \mcO_{S \times_\Z\A^1_\Z} \arrow[equal]{d} & \ar[l,"{\alpha_1}"']
    \mcM_1 \ar[d,"{\exists!\varphi}"]\ar[r,"\theta_1"] &
    i_*i^{-1}\mcM_1 &[0.5cm]
    i_*i^{-1}p^*\mcM_S \ar[l,"\sim","i_*i^{-1}\tilde{p}_1^{\flat}"'] \arrow[equal]{d} \\
    \mcO_{S \times_\Z \A^1_\Z} & \ar[l,"{\alpha_2}"']
    \mcM_2 \ar[r,"\theta_2"] &
    i_*i^{-1}\mcM_2 & \ar[l,"\sim","i_*i^{-1}\tilde{p}_2^{\flat}"']
    i_*i^{-1}p^*\mcM_S,
  \end{tikzcd}\]
  where, for each \(k\in \{1,2\}\), \(\tilde{p}_k^{\flat}: p^*\mcM_S\to \mcM_k\) is the morphism of sheaves of monoids on \(S\times_{\Z}\A^1_{\Z}\) that arises from the morphism of log schemes \(\tilde{p}_k^{\log}:(S\times_{\Z}\A^1_{\Z},\mcM_k,\alpha_k)\to S^{\log}\), and \(\theta_k:\mcM_k\to i_*i^{-1}\mcM_k\) is the natural morphism. 
  Then it follows immediately from various definitions involved that for each \(k\in \{1,2\}\), the natural morphism \(\theta_k:\mcM_k\to i_*i^{-1}\mcM_k\) is injective.
  Thus, to prove that the morphism of sheaves of monoids \(\varphi: \mcM_1\to \mcM_2\) is an isomorphism, i.e., to prove that the morphism of sheaves of monoids \(\varphi: \mcM_1\to \mcM_2\) is surjective, we may assume without loss of generality that \(S\) is isomorphic to the spectrum of an algebraically closed field.
  But then \autoref{lem: A1log is push log str} follows immediately.
\end{proof}

\begin{defi}
  We shall say that a morphism of log schemes \(f^{\log}\) is \textbf{log-like} if the underlying morphism of schemes is an isomorphism (cf. the terminology of \cite{Mzk04}).
\end{defi}

\begin{lem}\label{log like factorization}
  Let \(S^{\log}\) be a locally Noetherian fs log scheme and \(f^{\log}: X^{\log}\to Y^{\log}\) a morphism in \(\LSchb{S}\).
  Then the following assertions hold:
  \begin{enumerate}
    \item \label{exists factorization}
    (cf. \cite[Proposition 1.11 (i)]{Mzk15}).
    There exists a factorization
    \[X^{\log} \xrightarrow{g^{\log}} Z^{\log} \xrightarrow{h^{\log}} Y^{\log}\]
    of \(f^{\log}\) in \(\LSchb{S}\) such that \(g^{\log}\) is log-like, and \(h^{\log}\) is strict.
    \item \label{uniqueness factorization}
    (cf. \cite[Proposition 1.11 (ii)]{Mzk15}).
    The factorization \(X^{\log} \xrightarrow{g^{\log}} Z^{\log} \xrightarrow{h^{\log}} Y^{\log}\) of \ref{exists factorization} may be characterized up to a unique isomorphism, via the following universal property:
    The morphism \(h^{\log}\) is strict, and moreover, if
    \[X^{\log} \xrightarrow{g_0^{\log}} Z_0^{\log} \xrightarrow{h_0^{\log}} Y^{\log}\]
    is a factorization of \(f^{\log}\) in \(\LSchb{S}\) such that \(h_0^{\log}\) is strict, then there exists a unique morphism \(r^{\log}:Z^{\log}\to Z_0^{\log}\) in \(\LSchb{S}\) such that \(g_0^{\log} = r^{\log}\circ g^{\log}\), and \(h_0^{\log} \circ r^{\log} = h^{\log}\):
    \[\begin{tikzcd}
      X^{\log} \ar[r,"g^{\log}", "\text{\rm log-like}"'] \arrow[equal]{d} &[0.8cm]
      Z^{\log} \ar[r,"h^{\log}", "\text{\rm strict}"'] \ar[d,"\exists! r^{\log}"'] &[0.8cm]
      Y^{\log} \arrow[equal]{d} \\
      X^{\log} \ar[r,"g_0^{\log}"] &
      Z_0^{\log} \ar[r,"h_0^{\log}","\text{\rm strict}"'] &
      Y^{\log}.
    \end{tikzcd}\]
    \item \label{fanctorial factorization}
    Let
    \[\begin{tikzcd}
      X^{\log} \ar[r,"f^{\log}"]
      \ar[d,"p^{\log}"'] &
      Y^{\log} \ar[d,"q^{\log}"] \\
      X_0^{\log} \ar[r,"f_0^{\log}"] &
      Y_0^{\log}
    \end{tikzcd}\]
    be a commutative diagram in \(\LSchb{S}\);
    \[X^{\log} \xrightarrow{g^{\log}} Z^{\log} \xrightarrow{h^{\log}} Y^{\log}\]
    a factorization of \(f^{\log}\) such that \(g^{\log}\) is log-like, and \(h^{\log}\) is strict;
    \[X_0^{\log} \xrightarrow{g_0^{\log}} Z_0^{\log} \xrightarrow{h_0^{\log}} Y_0^{\log}\]
    is a factorization of \(f_0^{\log}\) in \(\LSchb{S}\) such that \(h_0^{\log}\) is strict.
    Then there exists a unique morphism \(r^{\log}:Z^{\log}\to Z_0^{\log}\) in \(\LSchb{S}\) such that \(g_0^{\log}\circ p^{\log} = r^{\log}\circ g^{\log}\), and \(h_0^{\log} \circ r^{\log} = q^{\log}\circ h^{\log}\):
    \[\begin{tikzcd}
      X^{\log} \ar[r,"g^{\log}", "\text{\rm log-like}"'] \ar[d,"p^{\log}"'] &[0.8cm]
      Z^{\log} \ar[r,"h^{\log}","\text{\rm strict}"'] \ar[d,"\exists! r^{\log}"'] &[0.8cm]
      Y^{\log} \ar[d,"q^{\log}"] \\
      X_0^{\log} \ar[r,"g_0^{\log}"] &
      Z_0^{\log} \ar[r,"h_0^{\log}","\text{\rm strict}"'] &
      Y_0^{\log}.
    \end{tikzcd}\]
  \end{enumerate}
\end{lem}

\begin{proof}
  Assertions \ref{exists factorization} and \ref{uniqueness factorization} follow immediately by considering the pull-back of the log structure of \(Y^{\log}\) to \(X\) via \(f\).
  Since
  \[X^{\log}\xrightarrow{(g_0^{\log}\circ p^{\log},f^{\log})} Z_0^{\log}\times_{Y_0^{\log}} Y^{\log}\xrightarrow{h_0^{\log}\times \id_{Y^{\log}}} Y^{\log},\]
  (where we note that the fiber product is taken in \(\Sch_{/S^{\log}}^{\log}\))
  is a factorization of \(f^{\log}\), and \(h_0^{\log}\times\id_{Y^{\log}}\) is strict,
  the existence and uniqueness portions of assertion \ref{fanctorial factorization} follow, respectively, by purely formal considerations, from the existence and uniqueness portions of assertion \ref{uniqueness factorization} (applied to \(\Sch_{/S^{\log}}^{\log}\)).
\end{proof}

\begin{cor}\label{log like}
  Let \(S^{\log}\) be a locally Noetherian fs log scheme and
  \(f^{\log}:X^{\log}\to Y^{\log}\) a morphism in \(\LSchb{S}\).
  Then \(f^{\log}\) is log-like if and only if
  the following condition holds:
  \begin{itemize}
    \item[ \ ] 
    For any factorization
    \[X^{\log}\xrightarrow{g^{\log}} Z_0^{\log}\xrightarrow{h^{\log}} Y^{\log}\]
    of \(f^{\log}\) in \(\LSchb{S}\) such that \(h^{\log}\) is strict, there exists a unique morphism \(r^{\log}: Y^{\log} \to Z_0^{\log}\) in \(\LSchb{S}\) such that \(g^{\log} = r^{\log}\circ f^{\log}\), and \(h^{\log}\circ r^{\log} = \id_{Y^{\log}}\):
    \[\begin{tikzcd}
      X^{\log} \ar[r,"f^{\log}"] \arrow[equal]{d} &[0.8cm]
      Y^{\log} \arrow[equal]{r} \ar[d,"\exists! r^{\log}"'] &[0.8cm]
      Y^{\log} \arrow[equal]{d} \\
      X^{\log} \ar[r,"g^{\log}"] &
      Z_0^{\log} \ar[r,"h^{\log}","\text{\rm strict}"'] &
      Y^{\log}.
    \end{tikzcd}\]
  \end{itemize}
  In particular,
  the property that \(f^{\log}\) is log-like
  may be characterized category-theoretically (cf. \autoref{strict general}) from the data \((\LSchb{S},f^{\log})\).
\end{cor}

\begin{proof}
  \autoref{log like} follows immediately from \autoref{log like factorization} \ref{exists factorization} \ref{uniqueness factorization}.
\end{proof}

Next, we give a characterization of the morphisms of monoid objects in \(\LSchb{S}\) that are isomorphic as objects of \(\Mor(\Mon(\LSchb{S}))\) to \(\alpha_{S^{\log},\A}^{\log}: S^{\log}\times_{\Z}\A_{\Z}^{1,\log} \to (S^{\log}\times_{\Z}\A_{\Z}^1,\times)\) (cf. \autoref{5-1}, \autoref{rem: A^1log is a monoid obj}).

\begin{prop}\label{monoid object A}
  Let
  \begin{itemize}
    \item \(S^{\log}\) be a locally Noetherian fs log scheme,
    \item \(A^{\log}\) a ring object of \(\LSchb{S}\) that is isomorphic as a ring object in \(\LSchb{S}\) to \(S^{\log}\times_{\Z}\A_{\Z}^1\), and
    \item \(\alpha^{\log}: M^{\log}\to (A^{\log},\times)\) a morphism of monoid objects in \(\LSchb{S}\).
  \end{itemize}
  Write \(A^{\times,\log}\hookrightarrow A^{\log}\) for the strict open immersion from the group of units of the ring object \(A^{\log}\in \LSchb{S}\) (where we note that since \(A^{\log}\cong S^{\log}\times_{\Z}\A^1\), it holds that \(A^{\times,\log}\cong S^{\log}\times_{\Z}\G_{m,\Z}\)).
  Then \(\alpha^{\log}: M^{\log}\to (A^{\log},\times)\) is isomorphic as an object of \(\Mor(\Mon(\LSchb{S}))\) to \[\alpha_{S^{\log},\A}^{\log}: S^{\log}\times_{\Z}\A_{\Z}^{1,\log} \to (S^{\log}\times_{\Z}\A_{\Z}^1,\times) \ \ \ \
  \text{(cf. \autoref{rem: A^1log is a monoid obj})}\]
  if and only if the following conditions hold:
  \begin{enumerate}
    \item \label{alpha log like}
    \(\alpha^{\log}\) is log-like.
    \item \label{alpha strict outside 0}
    The natural projection \(A^{\times,\log}\btimes_{A^{\log}}M^{\log}\xrightarrow{\sim} A^{\times,\log}\) is an isomorphism (where we note that by \autoref{lem: fiber prod}, the fiber product \(A^{\times,\log}\btimes_{A^{\log}}M^{\log}\) exists in \(\LSchb{S}\)).
    Write \(i_M^{\log}: A^{\times,\log}\hookrightarrow M^{\log}\) for the strict open immersion over \(A^{\log}\) determined by the natural projections \(A^{\times,\log} \xleftarrow{\sim} A^{\times,\log} \btimes_{A^{\log}} M^{\log} \hookrightarrow M^{\log}\).
    \item \label{alpha initial}
    For any log-like morphism \(f^{\log}: X^{\log}\to A^{\log}\) in \(\LSchb{S}\) such that the natural projection \(A^{\times,\log} \btimes_{A^{\log}} X^{\log} \xrightarrow{\sim} A^{\times,\log}\) is an isomorphism (where we note that by \autoref{lem: fiber prod}, the fiber product \(A^{\times,\log}\btimes_{A^{\log}}X^{\log}\) exists in \(\LSchb{S}\)), there exists a unique morphism \(g^{\log}: M^{\log}\to X^{\log}\) in \(\LSchb{S}\) such that \(f^{\log}\circ g^{\log} = \alpha^{\log}\), and \(i_X^{\log} = g^{\log}\circ i_M^{\log}\):
    \[\begin{tikzcd}
      A^{\times,\log} \ar[r,hookrightarrow,"i_X^{\log}"]
      \ar[d,hookrightarrow,"i_M^{\log}"'] &[1cm]
      X^{\log} \ar[d,"f^{\log}"] \\
      M^{\log} \ar[r,"\alpha^{\log}"]
      \ar[ru,"\exists !g^{\log}"]& A^{\log},
    \end{tikzcd}\]
    where \(i_X^{\log}:A^{\times,\log}\hookrightarrow X^{\log}\) is the strict open immersion over \(A^{\log}\) determined by the natural projections \(A^{\times,\log} \xleftarrow{\sim} A^{\times,\log} \btimes_{A^{\log}} X^{\log} \hookrightarrow X^{\log}\).
  \end{enumerate}
\end{prop}

\begin{proof}
  Necessity follows immediately from \autoref{lem: fiber prod} and \autoref{lem: A1log is push log str}.
  In the remainder of the proof of \autoref{monoid object A}, we prove sufficiency.
  Assume that conditions \ref{alpha log like}, \ref{alpha strict outside 0}, and \ref{alpha initial} hold.
  Let \[\alpha_0^{\log}: M_0^{\log}\to (A^{\log},\times)\] be a morphism of monoid objects in \(\LSchb{S}\) that is isomorphic as an object of \(\Mor(\Mon(\LSchb{S}))\) to \(\alpha_{S,\A}^{\log}: S^{\log}\times_{\Z}\A_{\Z}^{1,\log} \to (S^{\log}\times_{\Z}\A_{\Z}^1,\times)\).
  Write \(i_0^{\log}: A^{\times,\log}\hookrightarrow M_0^{\log}\) for the strict open immersion determined by the natural projections \(A^{\times,\log} \xleftarrow{\sim} A^{\times,\log} \btimes_{A^{\log}} M_0^{\log} \hookrightarrow M_0^{\log}\).
  By condition \ref{alpha initial}, we obtain a unique morphism \(g^{\log}: M^{\log}\to M_0^{\log}\) such that \(\alpha_0^{\log}\circ g^{\log} = \alpha^{\log}\), and \(i_0^{\log} = g^{\log}\circ i_M^{\log}\).
  On the other hand, by conditions \ref{alpha log like} \ref{alpha strict outside 0} and the necessity portion of \autoref{monoid object A}, we obtain a unique morphism \(g_0^{\log}: M_0^{\log}\to M^{\log}\) such that \(\alpha^{\log}\circ g_0^{\log} = \alpha_0^{\log}\), and \(i_M^{\log} = g_0^{\log}\circ i_0^{\log}\).
  Then
  \begin{align*}
    &\alpha^{\log} = \alpha_0^{\log}\circ g^{\log} = \alpha^{\log}\circ g_0^{\log} \circ g^{\log},
    &&\alpha_0^{\log} = \alpha^{\log}\circ g_0^{\log} = \alpha_0^{\log}\circ g^{\log} \circ g_0^{\log}, \\
    &i_M^{\log} = g_0^{\log}\circ i_0^{\log} = g_0^{\log}\circ g^{\log}\circ i_M^{\log},
    &&i_0^{\log} = g^{\log}\circ i_M^{\log} = g^{\log}\circ g_0^{\log}\circ i_0^{\log}.
  \end{align*}
  By the uniqueness portion of condition \ref{alpha initial}, \(g_0^{\log}\circ g^{\log} = \id_{M^{\log}}\).
  In a similar vein, by condition (iii) and the necessity portion of \autoref{monoid object A}, \(g^{\log}\circ g_0^{\log} = \id_{M_0^{\log}}\).
  Thus the underlying morphism of log schemes \(\alpha^{\log}: M^{\log}\to A^{\log}\) is isomorphic as an object of \(\Mor(\LSchb{S})\) to \(\alpha_{S^{\log},\A}^{\log}: S^{\log}\times_{\Z}\A_{\Z}^{1,\log} \to S^{\log}\times_{\Z}\A_{\Z}^1\).
  Finally, since \(i_M^{\log} = g_0^{\log}\circ i_0^{\log}\), the isomorphism \(g_0^{\log}\) is compatible with the monoid object structures on \(M^{\log}\) and \(M_0^{\log}\) in \(\LSchb{S}\).
  This completes the proof of \autoref{monoid object A}.
\end{proof}

Next, we summarize the properties of fs log schemes and morphisms of fs log schemes that were characterized category-theoretically in the present paper.


\begin{cor}\label{summarize}
  Let \(S^{\log}, T^{\log}\) be locally Noetherian fs log schemes, \[\bbullet, \bbbullet\subset \{\mathrm{red, qcpt, qsep, sep, ft}\}\] [possibly empty] subsets, and \(F:\LSchb{S}\xrightarrow{\sim} \LSchc{T}\) an equivalence of categories.
  Then the following assertions hold:
  \begin{enumerate}
    \item \label{properties preserve}
    Let \(f^{\log}: Y^{\log}\to X^{\log}\) be a morphism in \(\LSchb{S}\).
    Then the following assertions hold:
    \begin{enumerate}[label=(i-\alph*)]
      \item \label{enumi: main cor 01}
      \(X^{\log}\) is an fs log point if and only if \(F(X^{\log})\) is an fs log point.
      \item \label{enumi: main cor 02}
      \(f^{\log}\) is strict if and only if \(F(f^{\log})\) is strict.
      \item \label{enumi: main cor 03}
      \(f^{\log}\) is log-like if and only if \(F(f^{\log})\) is log-like.
    \end{enumerate}
    \item \label{A1 log preserve}
    Let \(A^{\log}\) be a ring object in \(\LSchb{S}\) and \(\alpha^{\log}:M^{\log} \to (A^{\log},\times)\) a morphism of monoid objects in \(\LSchb{S}\).
    Assume that
    \begin{itemize}
      \item
      \(A^{\log}\) is isomorphic as a ring object of \(\LSchb{S}\) to \(S^{\log}\times_{\Z}\A^1_{\Z}\), and
      \item
      \(F(A^{\log})\) is isomorphic as a ring object of \(\LSchc{T}\) to \(T^{\log}\times_{\Z}\A^1_{\Z}\).
    \end{itemize}
    Then \(\alpha^{\log}\) is isomorphic as an object of \(\Mor(\Mon(\LSchb{S}))\) to
    \[\alpha_{S,\A}^{\log}:S^{\log}\times_{\Z}\A_{\Z}^{1,\log} \to (S^{\log}\times_{\Z}\A^1_{\Z},\times) \ \ \ \ \ \text{(cf. \autoref{rem: A^1log is a monoid obj})}\]
    if and only if \(F(\alpha^{\log})\) is isomorphic as an object of \(\Mor(\Mon(\LSchc{T}))\) to
    \[\alpha_{T,\A}^{\log}:T^{\log}\times_{\Z}\A_{\Z}^{1,\log} \to (T^{\log}\times_{\Z}\A^1_{\Z},\times).\]
  \end{enumerate}
\end{cor}

\begin{proof}
  Assertion \ref{enumi: main cor 01} follows immediately from \autoref{prop: fs log pt}.
  Assertion \ref{enumi: main cor 02} follows immediately from \autoref{strict general}.
  Assertion \ref{enumi: main cor 03} follows immediately from \autoref{log like}.
  Assertion \ref{A1 log preserve} follows immediately from assertion \ref{enumi: main cor 03} and \autoref{monoid object A}.
\end{proof}

Let us recall that \(\bbullet/S^{\log}\) is a set of properties of (\(\univ{U}\)-small) schemes over the underlying scheme of \(S^{\log}\) (cf. \hyperlink{notations}{Notations} \hyperlink{notations}{and} \hyperlink{notations}{Conventions} --- \hyperlink{notations and conventions -- properties}{Properties of Schemes and Log Schemes}).
For any (\(\univ{U}\)-small) scheme \(S\), we shall write \(\Sch_{/S}\) for the category of (\(\univ{U}\)-small) \(S\)-schemes and \(\Schb{S}\subset \Sch_{/S}\) for the full subcategory of the objects of \(\Schb{S}\) that satisfy every property contained in \(\bbullet/S\).

Finally, we prove the main result of the present paper.

\begin{thm}\label{thm: functorial reconstruction general}
  Let \(S^{\log}, T^{\log}\) be locally Noetherian fs log schemes, \[\bbullet, \bbbullet\subset \{\mathrm{red, qcpt, qsep, sep, ft}\}\] [possibly empty] subsets, \(F:\LSchb{S}\xrightarrow{\sim} \LSchc{T}\) an equivalence of categories, and \(X^{\log}\in \LSchb{S}\) an object.
  Assume that the following condition holds:
  \begin{enumerate}[label=(\fnsymbol*),start=2]
    \item \label{functorial reconstruction condition}
    For any equivalence \(\underline{F}:\Schb{S} \xrightarrow{\sim} \Schc{T}\) and object \(X\in \Schb{S}\), there exists an isomorphism of schemes \(X\xrightarrow{\sim} \underline{F}(X)\) that is functorial with respect to \(X\in \Schb{S}\) (for fixed \(\underline{F}\)).
  \end{enumerate}
  Then the following assertions hold:
  \begin{enumerate}
    \item \label{functorial reconstruction}
    There exists an isomorphism of log schemes \(X^{\log}\xrightarrow{\sim} F(X^{\log})\) that is functorial with respect to \(X^{\log}\in \LSchb{S}\).
    \item \label{reconstruction of the base log scheme}
    Assume that \(\bbullet = \bbbullet\).
    Then there exists a unique isomorphism of log schemes \(S^{\log}\xrightarrow{\sim} T^{\log}\) such that \(F\) is isomorphic to the equivalence of categories \(\LSchb{S} \xrightarrow{\sim} \LSchb{T}\) induced by composing with this isomorphism of log schemes \(S^{\log}\xrightarrow{\sim} T^{\log}\).
  \end{enumerate}
\end{thm}

\begin{proof}
  First, we prove assertion \ref{functorial reconstruction}.
  For each \(*/Z\in \{\bbullet/S, \bbbullet/T\}\),
  write \(\LSch_{*/Z^{\log}}|_{\mathrm{schlk}}\subset \LSch_{*/Z^{\log}}\) for the full subcategory determined by the objects of \(\LSch_{*/Z^{\log}}|_{\mathrm{schlk}}\) whose structure morphism to \(Z^{\log}\) is strict.
  Then, by \autoref{summarize} \ref{enumi: main cor 02}, \(F\) induces an equivalence of categories
  \[
    \left(\Schb{S} \xrightarrow{\sim} \right) \ \LSchb{S}|_{\mathrm{schlk}} \xrightarrow[\sim]{F} \LSchc{T}|_{\mathrm{schlk}} \ \left(\xrightarrow{\sim} \Schc{T}\right).
  \]
  Thus, condition \ref{functorial reconstruction condition} implies that
  if we write \(\underline{F(X^{\log})}\) for the underlying scheme of the log scheme \(F(X^{\log})\), then
  it follows immediately from \autoref{log like factorization} \ref{exists factorization} \ref{fanctorial factorization} that we obtain an isomorphism of schemes
  \[X\xrightarrow{\sim} \underline{F(X^{\log})}\]
  that is functorial with respect to \(X^{\log}\in \LSchb{S}\).
  Hence, in particular,
  for any ring object \(A^{\log}\in \LSchb{S}\), it holds that
  \begin{itemize}
    \item[ \ ]
    \(A^{\log}\) is isomorphic as a ring object of \(\LSchb{S}\) to \(S^{\log}\times_{\Z}\A^1_{\Z}\) if and only if
    \item[ \ ]
    \(F(A^{\log})\) is isomorphic as a ring object of \(\LSchc{T}\) to \(T^{\log}\times_{\Z}\A^1_{\Z}\).
  \end{itemize}
  Since the map between sets of \((-)\)-valued points of \(S^{\log}\times_{\Z}\A_{\Z}^{1,\log}\) and \(S^{\log}\times_{\Z}\A_{\Z}^1\) induced by composing with \(\alpha_{S,\A}^{\log}: S^{\log}\times_{\Z}\A_{\Z}^{1,\log}\to (S^{\log}\times_{\Z}\A_{\Z}^1,\times)\) may be naturally identified (cf. \autoref{rem: A^1log is a monoid obj}) with the morphism between sheaves of monoids that defines the log structure on ``\((-)\)'', it follows from \autoref{summarize} \ref{A1 log preserve} that assertion \ref{functorial reconstruction} holds.

  Next, we prove the existence portion of assertion \ref{reconstruction of the base log scheme}.
  Write \[\left\{\varphi_{X^{\log}}: X^{\log}\xrightarrow{\sim} F(X^{\log})\right\}_{X^{\log}\in \LSchb{S}}\] for the family of functorial isomorphisms of log schemes discussed in assertion \ref{functorial reconstruction} and \(p^{\log}:F(S^{\log}) \xrightarrow{\sim} T^{\log}\) for the structure morphism in \(\LSchb{T}\).
  Then, it follows immediately from the functoriality of the family \(\{\varphi_{X^{\log}}\}_{X^{\log}\in \LSchb{S}}\) that \(F\) is isomorphic to the equivalence obtained by composing with \(\varphi_{S^{\log}}\circ p^{\log}: S^{\log}\xrightarrow{\sim} T^{\log}\).

  Finally, we prove the uniqueness portion of assertion \ref{reconstruction of the base log scheme}.
  Let \(f^{\log}, g^{\log}: S^{\log}\xrightarrow{\sim} T^{\log}\) be isomorphisms.
  For any object \(X^{\log}\in \LSchb{S}\), write \(p_{X^{\log}}^{\log}: X^{\log}\to S^{\log}\) for the structure morphism.
  Let \(\psi = \{ \psi_{X^{\log}}^{\log}: X^{\log}\xrightarrow{\sim} X^{\log}\}_{X^{\log}\in \LSchb{S}}\) be a family of isomorphisms of log schemes such that \(f^{\log}\circ p_{X^{\log}}^{\log} = g^{\log}\circ p_{X^{\log}}^{\log} \circ \psi_{X^{\log}}^{\log}\), and \(\psi_{X^{\log}}^{\log}\) is functorial with respect to \(X^{\log}\in \LSchb{S}\).
  Then \(f^{\log} = g^{\log}\circ \psi_{S^{\log}}^{\log}\).
  By the functoriality of \(\psi\), for any strict open immersion \(p_{U^{\log}}^{\log}: U^{\log} \hookrightarrow S^{\log}\), it holds that \(\psi_{S^{\log}}^{\log} \circ p_{U^{\log}}^{\log} = p_{U^{\log}}^{\log} \circ \psi_{U^{\log}}^{\log}\).
  Hence, in particular, the underlying morphism of topological spaces of \(\psi_{S^{\log}}^{\log}\) is equal to \(\id_{|S|}\) (cf. \cite[Theorem 7.24]{Johnstone}).
  By the functoriality of \(\psi\), for each \(i\in \{0,1,\infty\}\), it holds that \((\psi_{S^{\log}\times_{\Z}\P^1_{\Z}}^{\log}) \circ (\id_{S^{\log}} \times i) = (\id_{S^{\log}} \times i) \circ \psi_{S^{\log}}^{\log}\) (where we regard \(i\) as the morphism of schemes \(\Spec(\Z)\to \P^1_{\Z}\) corresponding to \(i\in \Z\cup\{\infty\}\)):
  \[\begin{tikzcd}
    S^{\log}
    \ar[r,"{\id_{S^{\log}}\times i}"]
    \ar[d,"\psi_{S^{\log}}^{\log}"',"\rotatebox{90}{\(\sim\)}"]
    & [1cm]
    S^{\log}\times_{\Z} \P^1_{\Z}
    \ar[d,"\psi_{S^{\log}\times_{\Z}\P^1_{\Z}}^{\log}", "\rotatebox{90}{\(\sim\)}"']
    \ar[r,hookleftarrow]
    &
    S^{\log}\times_{\Z}\A^1_{\Z}
    \ar[d, "\psi^{\log}_{S^{\log}\times_{\Z}\A^1_{\Z}}", "\rotatebox{90}{\(\sim\)}"']
    \\
    S^{\log} \ar[r,"{\id_{S^{\log}}\times i}"]
    & S^{\log}\times_{\Z} \P^1_{\Z}
    \ar[r,hookleftarrow]
    &
    S^{\log}\times_{\Z}\A^1_{\Z}.
  \end{tikzcd}\]
  Hence \(\psi_{S^{\log}\times_{\Z}\A^1_{\Z}}^{\log} = \psi_{S^{\log}}^{\log}\times \id_{\A^1_{\Z}}\).
  In particular, by the functoriality of \(\psi\), for any object \(X^{\log}\in \LSchb{S}\) and any element \(s\in \Gamma(X,\OX)\), if we write \(\tilde{s}^{\log}: X^{\log}\to S^{\log}\times_{\Z}\A^1_{\Z}\) for the morphism in \(\LSchb{S}\) corresponding to \(s\), then \(\tilde{s}^{\log}\circ \psi_{X^{\log}}^{\log} = (\psi_{S^{\log}}^{\log}\times \id_{\A^1_{\Z}}) \circ \tilde{s}^{\log}\).
  This implies that \(\psi_{X^{\log}}^{\#}(X)(s) = s \in \Gamma(X,\OX)\), and hence that the underlying morphism of schemes of \(\psi_{S^{\log}}^{\log}\) is equal to \(\id_S\).
  Next, observe that by the functoriality of \(\psi\), it follows from \autoref{lem: A1log is push log str} that \(\psi_{S^{\log}\times_{\Z}\A^{1,\log}_{\Z}}^{\log} = \psi_{S^{\log}}^{\log} \times \id_{\A^{1,\log}_{\Z}}\).
  In particular, by the functoriality of \(\psi\), for any object \(Y^{\log}\in \LSchb{S}\) and any element \(t\in \Gamma(Y,\mcM_Y)\), if we write \(\tilde{t}^{\log}: Y^{\log}\to S^{\log}\times_{\Z}\A^{1,\log}_{\Z}\) for the morphism in \(\LSchb{S}\) corresponding to \(t\), then \(\tilde{t}^{\log}\circ \psi_{Y^{\log}}^{\log} = (\psi_{S^{\log}}^{\log}\times \id_{\A^{1,\log}_{\Z}}) \circ \tilde{t}^{\log}\).
  This implies that \(\psi_{Y^{\log}}^{\flat}(Y)(t) = t\in \Gamma(Y,\mcM_Y)\), and hence that \(\psi_{S^{\log}}^{\log} = \id_{S^{\log}}\).
  Thus \(f^{\log} = g^{\log}\).
  This completes the proof of \autoref{thm: functorial reconstruction general}.
\end{proof}

\begin{thm}\label{last cor}
  Let \(S^{\log}, T^{\log}\) be locally Noetherian normal fs log schemes, \[\bbullet, \bbbullet \subset \{\mathrm{red, qcpt, qsep, sep, ft}\}\] [possibly empty] subsets, and \(F:\LSchb{S}\xrightarrow{\sim} \LSchc{T}\) an equivalence of categories.
  Assume that one of the following conditions \ref{last thm situation YJ}, \ref{last thm situation Mzk} holds:
  \begin{enumerate}[label=(\Alph*)]
    \item \label{last thm situation YJ}
    \(\bbullet, \bbbullet \subset \{\mathrm{red,qcpt,qsep,sep}\}\), and
    the underlying schemes of \(S^{\log}\) and \(T^{\log}\) are normal.
    \item \label{last thm situation Mzk}
    \(\bbullet = \bbbullet = \{\mathrm{ft}\}\).
  \end{enumerate}
  Then the following assertions hold:
  \begin{enumerate}
    \item \label{main 1}
    Let \(X^{\log}\in \LSchb{S}\) be an object.
    Then there exists an isomorphism of log schemes \(X^{\log}\xrightarrow{\sim} F(X^{\log})\) that is functorial with respect to \(X^{\log}\in \LSchb{S}\).
    \item \label{main 2}
    Assume that \(\bbullet = \bbbullet\).
    Then there exists a unique isomorphism of log schemes \(S^{\log}\xrightarrow{\sim} T^{\log}\) such that \(F\) is isomorphic to the equivalence of categories \(\LSchb{S} \xrightarrow{\sim} \LSchb{T}\) induced by composing with this isomorphism of log schemes \(S^{\log}\xrightarrow{\sim} T^{\log}\).
  \end{enumerate}
\end{thm}

\begin{proof}
  If condition \ref{last thm situation YJ} holds, then it follows immediately from 
  \autoref{thm: functorial reconstruction general} \ref{functorial reconstruction} \ref{reconstruction of the base log scheme} and \cite[Corollary 6.24]{YJ} that assertions \ref{main 1} and \ref{main 2} hold.
  If condition \ref{last thm situation Mzk} holds, then it follows immediately from \autoref{thm: functorial reconstruction general} \ref{functorial reconstruction} \ref{reconstruction of the base log scheme} and (the proof of) \cite[Theorem 1.7 (ii)]{Mzk04} that assertions \ref{main 1} and \ref{main 2} hold.
  This completes the proof of \autoref{last cor}.
\end{proof}

\begin{cor}
  \label{last cor A}
  Let \(S^{\log}, T^{\log}\) be 
  locally Noetherian normal fs log schemes and \[\bbullet, \bbbullet \subset \{\mathrm{red, qcpt, qsep, sep}\}\] subsets such that \(\{\mathrm{qsep, sep}\}\not\subset \bbullet\), and \(\{\mathrm{qsep, sep}\}\not\subset \bbbullet\).
  If the categories \(\LSchb{S}\) and \(\LSchc{T}\) are equivalent, then \(\bbullet = \bbbullet\), and \(S^{\log}\cong T^{\log}\).
\end{cor}

\begin{proof}
  \autoref{last cor A} follows immediately from \autoref{bbul bbbul} and \autoref{last cor} \ref{main 2}.
\end{proof}

\appendix

\section{A Lemma of Nakayama}
\label{appendix Nak Lem}

In this appendix, we prove an extension of \cite[Lemma 2.2.6 (i)]{Nak} (cf. \autoref{lem: qint Nak in other words}) to the case where
\(M\), \(L\) are quasi-integral, and \(N\) is an arbitrary sharp monoid (cf. \autoref{Nak extension}).


\begin{lem}
  \label{lem: quot is surj}
  Let \(f:N\to M\) and \(g:N\to L\) be morphisms of (arbitrary) monoids.
  Write \(P\dfn M\amalg_NL\); \(i_M: M\to P\) and \(i_L:L\to P\) for the natural inclusions;
  \(\pi:M\times L\to P\) for the morphism determined by \(i_M\) and \(i_L\).
  Then \(\pi\) is surjective.
\end{lem}

\begin{proof}
  Write \(P'\dfn \im(\pi)\), \(i:P'\hookrightarrow P\) for the inclusion morphism, and for each \(*\in \{M,L\}\), \(i_*':*\to P'\) for the morphism induced by \(i_*\).
  Then, by the universality of push-outs, there exists a unique morphism of monoids \(r:P\to P'\) such that \(r\circ i_M = i_M'\), and \(r\circ i_L = i_L'\).
  This implies that \(i\circ r = \id_P\).
  Thus \(i\) is surjective, which implies that \(\pi\) is surjective.
  This completes the proof of \autoref{lem: quot is surj}.
\end{proof}

\begin{lem}
  \label{lem: sharp fs pushout sharp}
  Let \(L,M,N\) be sharp monoids and
  \(f:N\to M\), \(g:N\to L\) local morphisms of monoids.
  Write \(P\dfn M\amalg_N L\);
  \(i_M:M\to P\) and \(i_L:L\to P\) for the natural inclusions;
  \(\pi:M\times L\to P\) for the morphism determined by \(i_M\) and \(i_L\).
  Then \(i_M\), \(i_L\), and \(\pi\) are local, and \(P\) is sharp.
\end{lem}

\begin{proof}
  Write \((\F_2,\times)\) for the underlying multiplicative monoid of the finite field \(\F_2 = \Z/2\Z\) (so \(\F_2^{\times}= \{1\}\)).
  For any monoid \(*\),
  write \(\varphi_*:* \to (\F_2,\times)\) for the unique local morphism of monoids.
  Since \(f\) and \(g\) are local,
  it holds that \(\varphi_M \circ f = \varphi_N = \varphi_L\circ g\).
  Hence there exists a unique morphism \(h:P\to (\F_2,\times)\) such that \(\varphi_M = h\circ i_M\), and \(\varphi_L = h\circ i_L\).
  Thus \(\varphi_{M\times L} = h\circ \pi\).
  Since \(\pi\) is surjective (cf. \autoref{lem: quot is surj}), \(h^{-1}(1) = 0\).
  This implies that \(P\) is sharp, that \(h = \varphi_P\), and that \(i_M\), \(i_L\), and \(\pi\) are local.
  This completes the proof of \autoref{lem: sharp fs pushout sharp}.
\end{proof}

\begin{defi}
  Let \(M\) be a monoid.
  Then we shall write \(\eta_M:M\to M^{\gp}\) for the natural morphism to the groupification.
\end{defi}

\begin{lem}
  \label{app lem: qint Nak in other words}
  Let \(L,M,N\) be sharp monoids and
  \(f:N\to M, g:N\to L\) local morphisms of monoids.
  Then the following assertions are equivalent:
  \begin{enumerate}
    \item \label{app Nak P is qint}
    \(M\amalg_NL\) is quasi-integral.
    \item \label{app Nak qint condition}
    For any elements \(m\in M\), \(l\in L\), and \(n\in N^{\gp}\),
    if \(f^{\gp}(n) = \eta_M(m)\), and \(-g^{\gp}(n) = \eta_L(l)\),
    then \(m = 0\), and \(l = 0\).
    \item \label{app Nak qint condition 3}
    \(M\), \(L\) are quasi-integral, and, moreover, for any element \(n\in N^{\gp}\), if \(f^{\gp}(n) \in M^{\inte}\), and \(-g^{\gp}(n) \in L^{\inte}\), then \(f^{\gp}(n) = 0\), and \(g^{\gp}(n) = 0\).
  \end{enumerate}
\end{lem}

\begin{proof}
  Write \(P\dfn M\amalg_NL\);
  \(i_M: M\to P\) and \(i_L:L\to P\) for the natural inclusions;
  \(\pi:M\times L\to P\) for the morphism determined by \(i_M\) and \(i_L\).

  To prove that \ref{app Nak P is qint} implies \ref{app Nak qint condition}, assume that \(P\) is quasi-integral.
  Let \(m\in M\), \(l\in L\), and \(n\in N^{\gp}\) be elements such that \(f^{\gp}(n) = \eta_M(m)\), and \(-g^{\gp}(n) = \eta_L(l)\).
  Then
  \[\eta_P(\pi(m,l)) = \pi^{\gp}(\eta_M(m),\eta_L(l)) = \pi^{\gp}(f^{\gp}(n),-g^{\gp}(n)) = 0.\]
  Since \(P\) is quasi-integral, \(\pi(m,l) = 0\).
  Moreover, since \(\pi\) is local (cf. \autoref{lem: sharp fs pushout sharp}), we conclude that \(m=0\), and \(l=0\).
  This completes the proof of the implication \ref{app Nak P is qint} \(\Rightarrow\) \ref{app Nak qint condition}.

  Next, we prove that \ref{app Nak qint condition} implies \ref{app Nak P is qint}.
  Assume that assertion \ref{app Nak qint condition} holds.
  Let \(p\in P\) be an element such that \(\eta_P(p) = 0\).
  Since \(\pi\) is surjective (cf. \autoref{lem: quot is surj}), there exist elements \(m\in M\) and \(l\in L\) such that \(p = \pi(m,l)\).
  Then
  \[\pi^{\gp}(\eta_M(m),\eta_L(l)) = \eta_P(\pi(m,l)) = \eta_P(p) = 0.\]
  Hence there exists an element \(n\in N^{\gp}\) such that \(\eta_M(m) = f^{\gp}(n)\), and \(\eta_L(l) = -g^{\gp}(n)\).
  Thus, by assertion \ref{app Nak qint condition}, it holds that \(m=0\), and \(l=0\).
  This implies that \(p=0\), i.e., that \(P\) is quasi-integral.
  This completes the proof of the implication \ref{app Nak qint condition} \(\Rightarrow\) \ref{app Nak P is qint}.

  Finally, by the definition of the notion of a quasi-integral monoid, assertion \ref{app Nak qint condition} is equivalent to assertion \ref{app Nak qint condition 3}.
  This completes the proof of \autoref{app lem: qint Nak in other words}.
\end{proof}

\begin{cor}[{cf. \cite[Lemma 2.2.6 (i)]{Nak}, \autoref{lem: qint Nak in other words}}]
  \label{Nak extension}
  Let \(L,M,N\) be sharp monoids and
  \(f:N\to M, g:N\to L\) local morphisms of monoids.
  Assume that \(M\) and \(L\) are quasi-integral.
  Then the following assertions are equivalent:
  \begin{enumerate}
    \item \label{Nak cor P is qint}
    \(M\amalg_NL\) is quasi-integral.
    \item \label{Nak cor qint condition}
    For any element \(n\in N^{\gp}\),
    if \(f^{\gp}(n) \in M^{\inte}\), and \(-g^{\gp}(n)\in L^{\inte}\),
    then \(f^{\gp}(n) = 0\), and \(g^{\gp}(n) = 0\).
  \end{enumerate}
\end{cor}

\begin{proof}
  \autoref{Nak extension} follows immediately from \autoref{app lem: qint Nak in other words} \ref{app Nak P is qint} \(\Leftrightarrow\) \ref{app Nak qint condition 3}.
\end{proof}

\end{document}